\documentclass[final,12pt]{colt2026} % Include author names
\DeclareMathOperator{\KL}{KL}
\newcommand{\KLinf}{\mathrm{KL}_{\inf}}

\newcommand{\Pbb}{\mathbb{P}}

\newcommand{\bbP}{\mathbf{P}}

\title[An Asymptotic Law of the Iterated Logarithm for $\KLinf$]{An Asymptotic Law of the Iterated Logarithm for $\KLinf$}
\usepackage{times}
\usepackage{amsmath,amssymb,mathtools}
\usepackage{comment}
\usepackage{hyperref}
\usepackage{graphicx}
\usepackage{subcaption}
\usepackage{booktabs}
\usepackage{enumitem}
\usepackage{mathrsfs}
\usepackage{comment}
\usepackage{siunitx}
\usepackage{xcolor}

 \coltauthor{\Name{Ashwin Ram} \Email{aram2@andrew.cmu.edu}\\
  \Name{Aaditya Ramdas} \Email{aramdas@cmu.edu}\\
\addr Carnegie Mellon University}

\begin{document}

\maketitle

\begin{abstract}%
The population $\KLinf$ is a fundamental quantity that appears in lower bounds for  (asymptotically) optimal regret of pure-exploration stochastic bandit algorithms, and optimal stopping time of sequential tests. Motivated by this, an empirical $\KLinf$ statistic is  frequently used in the design of (asymptotically) optimal bandit algorithms and sequential tests. While nonasymptotic concentration bounds for the empirical $\KLinf$ have been developed, their optimality in terms of constants and rates is questionable, and their generality is limited (usually to bounded observations). The fundamental limits of nonasymptotic concentration are often described by the asymptotic fluctuations of the statistics. With that motivation, this paper presents a tight (upper and lower) law of the iterated logarithm for empirical $\KLinf$ applying to extremely general (unbounded) data.
\end{abstract}


\newcommand{\what}{\widehat}
\section{Introduction}
Consider the typical setting of observing independent and identically distributed random variables $X_1, X_2,\dots$, with common distribution $\mathbf P\in\mathcal P([a,b])$ all with mean $\mu$ and variance $\sigma^2\in(0,\infty)$. Denote the empirical law of said random variables as $\what P_t=\frac1t\sum_{i=1}^t\delta_{X_i}$ with empirical mean $\what\mu_t$ and variance $\what\sigma_t^2$. Consider some distribution $\nu$ supported on a compact interval $[a, b]$ and a target mean $m\in[a, b]$. Given this, define the one-sided mean constrained information projection as,
\[
\KLinf(\nu, m):= \inf\Big\{\KL(\nu\|Q): Q\in\mathcal P([a,b]),\ \int x\,dQ(x)\ge m\Big\},
\]
and we take $\KL(\nu\|Q)=\infty$ if $\nu\not\ll Q$. In this paper, we characterize the almost-sure iterated logarithm scale of the empirical projection cost at the true mean, $\KLinf(\what P_t, \mu)$, where all probabilities and almost sure statements are under i.i.d.\ draws from $\mathbf P$. 

It is well known that the sample mean satisfies the Hartman-Wintner law of the iterated logarithm (LIL) \citep{khinchin1924, Kolmogoroff1929, hartman1941law, Feller1943, Chung1948, Strassen1964, Teicher1974}. That is, almost surely we have  $\limsup_{t\to\infty}\frac{\what\mu_t-\mu}{\sqrt{(2\sigma^2\log\log t)/t}}=1$ and $\liminf_{t\to\infty}\frac{\what\mu_t-\mu}{\sqrt{(2\sigma^2\log\log t)/t}}=-1$. Our main contribution, however, is that we show that $\KLinf(\what P_t, \mu)$ in fact has the same iterated logarithm constant (through expanding the $\KLinf$ quadratically). In other words, we show that $\KLinf(\what P_t,\mu)=(1+o(1))\frac{(\mu-\what\mu_t)_+^2}{2\what\sigma_t^2}$. As a consequence, with probability one, $\limsup_{t\to\infty}\frac{t\KLinf(\what P_t,\mu)}{\log\log t}=1$. We prove it through analyzing the affine tilt of $\what P_t$ to get a sharp upper bound and then a dual lower bound using the Donsker-Varadan inequality \citep{DonskerVaradhan1975, Csiszar2003} alongside uniform Taylor control. 

Importantly, we extend our sharp LIL equality to situations beyond the above introduced compact support by restricting the competing class to a slowly-growing deterministic envelope $[-B_t, B_t]$ and show that if a very weak sufficient condition is met (the data almost surely eventually satisfying $|X_i|\le B_t$ for all $i\le t$, for some sequence $B_t=o(\sqrt{t/\log\log t})$), then the same sharp-$1$ LIL constant holds for the time varying functional $\KLinf^{(t)}$ defined over $\mathcal P([-B_t, B_t])$. We show that such envelopes are immediate when analyzing typical tail ``regimes'' such as sub-Gaussian, sub-exponential, and finite $p$-th moment.

The fundamental mechanism behind KL-based index policies for stochastic multi-armed bandits are in fact the mean constrained KL projections 
% (denoted as $D_{\inf}$ or $K_{\inf}$) 
\citep{Robbins1952, Gittins1989, Lai1985, Burnetas1996, Auer2002, Bubeck2012, gariviercappe2011, cappe2013, hondatakemura2010, hondatakemura2015, kaufmann2012, Lattimore2020}. There has also been an enormous amount of work that uses large deviations and nonasymptotic concentration to derive logarithmic regret \citep{Cramer1938, Kullback1951, Sanov1957, Dembo1998, Boucheron2013, Howard2020, howard2021}. Instead, our work is the first to analyze (and characterize) the natural almost-sure fluctuation scale that the empirical projection has, $t\KLinf(\what P_t, \mu)$: we show it oscillates at order $\log\log t$ with a sharp constant. 

\paragraph{Related Work.}
Khinchin and Kolmogorov were among the first to analyze the LIL for sums of random variables and provide a finite-variance refinement of Hartman-Wintner \citep{khinchin1924, Kolmogoroff1929, hartman1941law}, and there certainly have been innumerable extensions and variants \citep{Feller1943, Chung1948, Strassen1964, Teicher1974, stout1970, stout1974, bingham1986, deacosta1983}. Especially important in sequential inference, there has been later work in generalizing LIL to finite-horizon or time-uniform settings through, for example, sharp finite-time martingale LIL bounds and confidence-sequences \citep{wald1947, darlingrobbins1967, balsubramani2015, Howard2020, howard2021}. Moreover, laws of the iterated logarithm have been proven to extend far beyond the sample mean itself, with laws of the logarithm and functional LILs for empirical and local empirical processes \citep{Shorack1986, vanderVaart1996, Deheuvels1994, Mason2004, Einmahl2004}. 

Fundamentally, the KL divergence \citep{Kullback1951} is the driving force, so to speak, of large deviations of empirical measures and entropy minimization \citep{Sanov1957, Dembo1998, Dupuis1997}. Csisz{\'a}r's $I$-divergence geometry and his refinements of them gave intuition for information projections under linear constraints \citep{Csiszar1975, Csiszar1984, Csiszar2003}. And closely related to all these entropy-tilting ideas is empirical likelihood and exponential tilting methods in statistics \citep{Owen1988, Owen2001}. Now, it is well understood in the bandit literature that optimality and bandit lower bounds are characterized by $\KL$ ``information'' defined by taking infima over so-called confusing alternatives \citep{Robbins1952, Gittins1989, Lai1985, Burnetas1996, kaufmann2012}. That is the motivation of a number of works that presented KL-based optimism indices and of course nonparametric variants to these \citep{Robbins1952, gariviercappe2011, cappe2013, hondatakemura2010, hondatakemura2015, Lattimore2020}. With all this said, we present a completely orthogonal contribution. These procedures frequently compute the empirical mean constrained projection cost as mentioned. To that end, we provide the first almost-sure LIL calibration of this $\KLinf$ cost.

The rest of this paper is organized as follows. In Section~\ref{sec:upper-law}, we provide an upper bound on the empirical $\KLinf$, which we then tighten (thereby showing equality) in Section~\ref{sec:LIL-KLinf-lower}. We present a theorem to extend these results to unbounded data in Section~\ref{sec:beyond-bounded}. In Appendix~\ref{sec:full-proofs}, we provide all the full proofs that were not provided in our main text. In Appendix~\ref{sec:further-applications-beyond-bounded} we give three concrete ``instantiations'' of our general theorem from Section~\ref{sec:beyond-bounded} holding. Finally, in Appendix~\ref{sec:env-tightness}, we prove the tightness of the sharp envelope introduced in Section~\ref{sec:beyond-bounded} that allows us to extend our results to unbounded data.
\newcommand{\Var}{\mathrm{Var}}
\newcommand{\Ebb}{\mathbb{E}}
\section{An Upper Law of Iterated Logarithm for the Empirical $\mathrm{KL}_{\inf}$}\label{sec:upper-law}
In this section, let us formalize the idea of $\KLinf$ satisfying an asymptotic LIL, in a general setting, outside of any specific application like sequential testing, etc. Recall our same setup as before such that $X_1, X_2, X_3, \dots$ are iid random variables with common distribution $\mathbf P\in\mathcal P([a,b])$ supported on the compact interval $[a,b]$, with mean $\mu:=\Ebb[X_1]$ and variance $\sigma^2:=\Var(X_1)\in(0,\infty)$. All probabilities $\bbP(\cdot)$ and expectations $\Ebb[\cdot]$ below are taken with respect to  $\mathbf P$.
% \footnote{Namely, the associated product measure on the sequence.} 
Given all this, let us denote the empirical measure for $t\ge 1$ as,
\[
\widehat P_t := \frac{1}{t}\sum_{i=1}^t \delta_{X_i},
\]
and let us denote the empirical mean and variance respectively as,
\[
\widehat\mu_t := \int x\,d\widehat P_t(x)=\frac{1}{t}\sum_{i=1}^t X_i, \qquad \widehat\sigma_t^2 := \int (x-\widehat\mu_t)^2\,d\widehat P_t(x)=\frac{1}{t}\sum_{i=1}^t (X_i-\widehat\mu_t)^2.
\]

\noindent Further, for any measurable $h$, we write $\Ebb_{\widehat P_t}[h(X)]:= \int h(x)d\widehat P_t(x) = \frac1t\sum_{i=1}^t h(X_i)$. Importantly, albeit $\widehat P_t$ indeed is random, conditioned on the data $(X_1, X_2, \dots)$ it is fixed. As a consequence, $\Ebb_{\widehat P_t}$ is simply an empirical average. Now, with those defined, let us explain how we denote the $\KLinf$ in this setting. Namely, for a probability measure $\nu$ taking support on $[a, b]$ and a scalar $m\in[a,b]$, let us define it as,
\[
\KL_{\inf}(\nu,m) := \inf\Big\{\KL(\nu\|Q): Q\in\mathcal P([a,b]),\ \int x\,dQ(x)\ge m\Big\},
\]
where $\KL(\nu\|Q)$ is just the typical KL divergence. And we adopt the convention that $\KL(\nu\|Q)=\infty$ if $\nu\not\ll Q$. In addition, one point we want to make is that the compact class that we analyze here does in fact matter. Why? Because if one allows $Q$ to place mass arbitrarily far on the upper tail,\footnote{As an example, on $\mathbb R$ with no restrictions.} then clearly $\KL_{\inf}(\nu, m)$ can collapse to $0$ whenever $m$ exceeds the mean of $\nu$. Meaning, one can easily ``sprinkle'' if you will an arbitrarily small amount of mass at some very large point to meet the mean constraint at a vanishing KL cost. As such, a bounded (or some otherwise constrained) model class is really what it takes to make the $\KLinf$ nontrivial. Hence that is precisely the motivation for our upper LIL theorem below.

\begin{theorem}\label{thm:LIL-KLinf}
Under our above setup we have that,
\[
\boxed{
\bbP\!\left( \limsup_{t\to\infty}\frac{t\,\KL_{\inf}(\widehat P_t,\mu)}{\log\log t}\le 1 \right)=1.
}
\]
Equivalently, for every $\varepsilon>0$, almost surely there exists a (possibly random) $T_\varepsilon<\infty$ such that for all $t\ge T_\varepsilon$,
\[
\KL_{\inf}(\widehat P_t,\mu)\le (1+\varepsilon)\frac{\log\log t}{t}.
\]
\end{theorem}

Before proceeding with our proof of this theorem, recall that in a sense the quantity $\KL_{\inf}(\widehat P_t,\mu)$ is the smallest relative-entropy budget needed to perturb the empirical law so that its mean can be at least the true mean $\mu$. If in fact the empirical mean $\widehat\mu_t \ge \mu$, then no perturbation of course would be necessary and $\KLinf(\widehat P_t,\mu)=0$. Now, in the case where $\widehat\mu_t <\mu$, one of the simplest ways to increase the mean is to slightly upweight observations above $\widehat\mu_t$ and downweight those that are below $\widehat\mu_t$. This must be done in a way that preserves the total mass. Such a perturbation is infinitesimal, so its KL cost is second-order: meaning, it is proportional to $\big(\mu-\widehat\mu_t\big)^2/(2\widehat\sigma_t^2)$. From there we can use the classic LIL for $\widehat\mu_t$ to convert immediately into an LIL for $\KLinf$. With all these being said, one more tool we will need for our theorem's proof will be the following lemma: a uniform Taylor bound for $-\log(1+u)$ on a shrinking neighborhood.

\begin{lemma}\label{lem:Taylor}
Take some $r\in(0,1)$ and  define $f(u):=-\log(1+u)$ on $(-1,\infty)$. Then, for every $u\in[-r, r]$, we have that \(f(u)\le -u + \frac{u^2}{2} + \frac{|u|^3}{3(1-r)^3}.\) 
\end{lemma}

\noindent With all of this in mind, we are now finally ready for our proof of Theorem~\ref{thm:LIL-KLinf}.
\begin{proof}[Theorem~\ref{thm:LIL-KLinf}]
Let us start by setting $\Delta_t := (\mu - \widehat\mu_t)_+ = \max\{0,\mu-\widehat\mu_t\}$. Note that if $\Delta_t=0$, then necessarily $\widehat\mu_t\ge \mu$. This means that the empirical law $\widehat P_t$ must itself satisfy the mean constraint. Thus, $\KLinf(\widehat P_t, \mu)=0$. So, only the times with $\Delta_t>0$ matter and hence that is exactly what we're going to focus on. Let us work on the almost sure event on which the strong law of large numbers and the LIL for $\widehat\mu_t$ both hold. Namely, $\widehat\mu_t\to\mu$ almost surely and $\widehat\sigma_t^2\to\sigma^2$ almost surely. Note that also the ``classical'' LIL by \cite{hartman1941law} implies that for iid sequences with $\Var(X_1)=\sigma^2\in(0,\infty)$,
\[
\limsup_{t\to\infty}\frac{\widehat\mu_t-\mu}{\sqrt{(2\sigma^2\log\log t)/t}}=1, \qquad \liminf_{t\to\infty}\frac{\widehat\mu_t-\mu}{\sqrt{(2\sigma^2\log\log t)/t}}=-1,
\]
almost surely. Now, let us define a feasible affine tilt of the empirical law. That is, for those $t$ with $\Delta_t>0$ and $\widehat\sigma_t^2>0$, define $\theta_t :=\frac{\Delta_t}{\widehat \sigma_t^2}$ and $\frac{dQ_t}{d\widehat P_t}(x):=1+\theta_t(x-\widehat\mu_t)$. Crucially, this is the simplest possible perturbation that preserves the total mass because $\int (x-\widehat\mu_t)\,d\widehat P_t(x)=0$. In addition, this perturbation pushes the mean upward at a rate proportional to the variance. Furthermore, because of the fact that $x\in[a,b]$, we have that $|x-\widehat\mu_t|\le (b-a)$, and so,
\[
\inf_{x\in[a,b]}\Big(1+\theta_t(x-\widehat\mu_t)\Big) \ge 1-\theta_t(b-a).
\]
Since $\Delta_t\to 0$ and $\widehat\sigma_t^2\to\sigma^2>0$, we have that $\theta_t(b-a)\to 0$ almost surely. Thus, this means that for all sufficiently large $t$ the Radon-Nikodym derivative is strictly positive on $[a,b]$. As such, for such large $t$, $Q_t$ is therefore a well-defined probability measure supported on $[a,b]$, and,
\[
\int dQ_t = \int \big(1+\theta_t(x-\widehat\mu_t)\big)\,d\widehat P_t(x) = 1 + \theta_t\underbrace{\int(x-\widehat\mu_t)\,d\widehat P_t(x)}_{=0}=1.
\]

\noindent And, moreover, we know that its mean is exactly lifted to $\mu$. That is,
\begin{align*}
\int x\,dQ_t(x) &= \int x\big(1+\theta_t(x-\widehat\mu_t)\big)\,d\widehat P_t(x)\\
&= \underbrace{\int xd\widehat P_t(x)}_{\widehat\mu_t} + \theta_t \int x(x-\widehat\mu_t)\,d\widehat P_t(x) = \widehat\mu_t + \theta_t \int \underbrace{x(x-\widehat\mu_t)}_{(x-\widehat\mu_t)^2 + \widehat\mu_t(x-\widehat\mu_t)}\,d\widehat P_t(x)\\ 
&= \widehat\mu_t + \theta_t \int (x-\widehat\mu_t)^2\,d\widehat P_t(x) + \underbrace{\theta_t \int \widehat\mu_t(x-\widehat\mu_t)\,d\widehat P_t(x)}_{=0} \\
&= \widehat\mu_t + \theta_t \int (x-\widehat\mu_t)^2\,d\widehat P_t(x)\\
&= \widehat\mu_t + \theta_t\,\widehat\sigma_t^2 = \widehat\mu_t + \Delta_t = \mu.
\end{align*}

\noindent Hence, gorgeously, $Q_t$ satisfies the constraint $\int x\,dQ_t(x)\ge \mu$ exactly. Therefore, it follows that \(\KL_{\inf}(\widehat P_t,\mu)\le \KL(\widehat P_t\|Q_t)\). We now need to bound the KL cost of the actual tilt. By definition we know that,
\[
\KL(\widehat P_t\|Q_t) = \int \log\!\left(\frac{d\widehat P_t}{dQ_t}\right)\,d\widehat P_t = - \int \log\!\Big(1+\theta_t(x-\widehat\mu_t)\Big)\,d\widehat P_t(x).
\]

\noindent Now, let $X\sim\widehat P_t$ and define the random variable $U_t:=\theta_t(X-\widehat\mu_t)$. Then, necessarily because we treat $\widehat P_t$ as fixed conditioned on $(X_1, \dots, X_t)$, $\Ebb_{\widehat P_t}[U_t]=\theta_t\Ebb_{\widehat P_t}[X-\widehat\mu_t]=\theta_t\left(\int xd\widehat P_t(x)-\widehat\mu_t\right)=0$. Furthermore, we know that $|U_t|\le r_t$, where \(r_t:=\theta_t(b-a)=\frac{\Delta_t}{\widehat\sigma_t^2}(b-a)\xrightarrow[t\to\infty]{}0\) almost surely. So, to recap, we learned two things about our random variable $U_t$. First, it is centered, and second, uniformly small. Now, let us take $t$ large enough so that $r_t\in(0,1)$ and then quickly apply Lemma~\ref{lem:Taylor} with $r=r_t$ pointwise to $u=U_t$. If we do this, we get that,
\[
-\log(1+U_t) \le -U_t + \frac{U_t^2}{2} + \frac{|U_t|^3}{3(1-r_t)^3}.
\]

\noindent Then, taking $\widehat P_t$-expectations and leveraging the fact that $\Ebb_{\widehat P_t}[U_t]=0$, we get that,
\begin{equation}\label{eq:KL-bound}
\KL(\widehat P_t\|Q_t) \le \frac{\Ebb_{\widehat P_t}[U_t^2]}{2} + \frac{\Ebb_{\widehat P_t}[|U_t|^3]}{3(1-r_t)^3},
\end{equation}
as the first term $-\Ebb_{\widehat P_t}[U_t]=0$. Let us now compute these moments and then plug them back in. First, for the second moment we know that $\Ebb_{\widehat P_t}[U_t^2]=\theta_t^2\,\Ebb_{\widehat P_t}[(X-\widehat \mu_t)^2]=\theta_t^2\,\widehat\sigma_t^2$. Second, for the third absolute moment we know that $|U_t|\le r_t$. So, it follows that \(\Ebb_{\widehat P_t}[|U_t|^3]\le r_t\,\Ebb_{\widehat P_t}[U_t^2]= r_t\,\theta_t^2\,\widehat\sigma_t^2\). Plugging all these back into \eqref{eq:KL-bound} gives us,
\[
\KL(\widehat P_t\|Q_t) \le \frac{\theta_t^2\,\widehat\sigma_t^2}{2}\left(1+\frac{2r_t}{3(1-r_t)^3}\right) = \frac{\Delta_t^2}{2\widehat\sigma_t^2}\left(1+\frac{2r_t}{3(1-r_t)^3}\right).
\]

\noindent Now, we know that $r_t\to 0$ almost surely. As such, the $(1+\frac{2r_t}{3(1-r_t)^3})$ term is just essentially $(1+o(1))$ almost surely. Hence, we get that,
\begin{equation}\label{eq:KL-asymp}
\KL_{\inf}(\widehat P_t,\mu)\le \KL(\widehat P_t\|Q_t) \le (1+o(1))\,\frac{\Delta_t^2}{2\widehat\sigma_t^2}, \quad \text{almost surely as $t\to\infty$.}
\end{equation}

\noindent This really is very interesting here. We are seeing a local quadratic behavior of the KL. Namely, that the KL cost is second-order in the mean deficit. This aside, let us now inject the classical LIL for the mean: in other words, let us convert the mean LIL into a $\KLinf$ LIL to finish. The classic LIL for $\widehat\mu_t$ of course tells us that almost surely, \(\limsup_{t\to\infty}\frac{(\widehat\mu_t-\mu)^2}{(2\sigma^2\log\log t)/t}=1\). Furthermore, since $(\mu-\widehat\mu_t)_+\le |\widehat\mu_t-\mu|$, we know that the same upper bound holds for $\Delta_t$. In addition, the classic LIL also gives us \(\liminf(\widehat\mu_t-\mu)/\sqrt{(2\sigma^2\log\log t)/t}=-1\). This means that along a subsequence we have that $\widehat\mu_t-\mu<0$ with asymptotically maximal magnitude. Naturally, then this forces with probability one that $\limsup_{t\to\infty}\frac{\Delta_t^2}{(2\sigma^2\log\log t)/t}=1$. Obviously also, we know that with probability one that $\widehat\sigma_t^2\to\sigma^2$. Hence,
\[
\limsup_{t\to\infty}\frac{t}{\log\log t}\cdot\frac{\Delta_t^2}{2\widehat\sigma_t^2} = \limsup_{t\to\infty} \left(\frac{\sigma^2}{\widehat\sigma_t^2}\right) \left(\frac{\Delta_t^2}{(2\sigma^2\log\log t)/t}\right) =1,\quad\text{almost surely.}
\]

\noindent To finish, all we need to do is combine this result with \eqref{eq:KL-asymp} to get us with probability one that,
\[
\limsup_{t\to\infty}\frac{t\,\KL_{\inf}(\widehat P_t,\mu)}{\log\log t}\le 1,
\]
which is exactly our desired LIL bound. So, we are done.
\end{proof}

\section{A Matching Lower Law of Iterated Logarithm for the Empirical $\mathrm{KL}_{\inf}$}\label{sec:LIL-KLinf-lower}
In this section, we are now going to show that, in addition to the upper bound we proved, in fact there exists a matching lower bound, hence indeed the $\KLinf$ satisfies an asymptotic LIL with equality. To recap, Theorem~\ref{thm:LIL-KLinf} established the upper bound with probability one that,
\[
\limsup_{t\to\infty}\frac{t\,\KL_{\inf}(\widehat P_t,\mu)}{\log\log t}\le 1.
\]

\noindent We're now going to prove a matching lower bound. To this end, the $\limsup$ constant will become exactly $1$. To sort of prelude this entire section, we state the main theorem first. 

\begin{theorem}\label{thm:LIL-KLinf-equality}
Under the same exact setup of Theorem~\ref{thm:LIL-KLinf} we have in fact that,
\[
\boxed{
\bbP\left(\limsup_{t\to\infty}\frac{t\,\KL_{\inf}(\widehat P_t,\mu)}{\log\log t}=1\right)=1.
}
\]
\end{theorem}

\noindent To prove this, the only additional item we need is the Donsker-Varadhan variational inequality for $\KL(\nu\|Q)$ \citep{DonskerVaradhan1975}, which we state below and provide a self-contained proof  in the appendix for completeness.  

\begin{lemma}\label{lem:DV-lower}
Let $(\Omega,\mathcal F)$ be a measurable space and let $\nu,Q$ be probability measures on the space. Now suppose that $\nu\ll Q$. Then, for every bounded and measurable $\varphi:\Omega\to\mathbb R$ one has,
\[
\KL(\nu\|Q) \ge \int \varphi\,d\nu - \log\left(\int e^{\varphi}dQ\right).
\]
If $\nu\not\ll Q$, then $\KL(\nu\|Q)=\infty$ by convention and hence the inequality is trivially true. 
\end{lemma}
\begin{comment}
\begin{lemma}\label{lem:Taylor-lower-log}
Consider a particular $r\in(0,1)$ and define the function $g(u):=\log(1+u)$ on $(-1,\infty)$. Then for every $u\in[-r, r]$ we have that,
\[
g(u) \ge u-\frac{u^2}{2} - \frac{|u|^3}{3(1-r)^3}.
\]
\end{lemma}
\end{comment}

\noindent We are now ready to prove  Theorem~\ref{thm:LIL-KLinf-equality}.
\begin{proof}[Theorem~\ref{thm:LIL-KLinf-equality}]
As we mentioned before, we already have an almost-sure upper bound from the first Theorem~\ref{thm:LIL-KLinf}. So it remains to prove the matching lower bound with probability one that,
\[
\limsup_{t\to\infty}\frac{t\KLinf(\widehat P_t, \mu)}{\log\log t}\ge 1.
\]
To begin, as in our proof of Theorem~\ref{thm:LIL-KLinf}, define $\Delta_t:=(\mu-\widehat\mu_t)_+=\max\{0,\mu-\widehat\mu_t\}$. We know that whenever $\Delta_t=0$, we have that $\widehat\mu_t\ge\mu$, which means that $\KLinf(\widehat P_t,\mu)=0$ by feasibility of $\widehat P_t$ alone. The takeaway is therefore that the lower bound for the $\limsup$ can only come from those times where $\Delta_t>0$. Let's work on the probability one event where both the strong law and classic LIL hold. Meaning, we will work where $\widehat\mu_t\to\mu$ and $\widehat\sigma_t^2\to\sigma^2\in(0,\infty)$. And, we will work where, $\limsup_{t\to\infty}\frac{\widehat\mu_t-\mu}{\sqrt{(2\sigma^2\log\log t)/t}}=1$ and $\liminf_{t\to\infty}\frac{\widehat\mu_t-\mu}{\sqrt{(2\sigma^2\log\log t)/t}}=-1$. Importantly, the second relation here tells us that along some subsequence $\{t_k\}$, we have that $\widehat\mu_{t_k}<\mu$ with asymptotically optimal magnitude. That is, we have with probability one that,
\begin{equation}\label{eq:Delta-LIL}
\limsup_{t\to\infty}\frac{\Delta_t^2}{(2\sigma^2\log\log t)/t}=1.
\end{equation}

\noindent Indeed if we let $a_t:=\sqrt{(2\sigma^2\log\log t)/t}$, the LIL gives us that $\limsup (\widehat\mu_t-\mu)/a_t=1$ and $\liminf(\widehat\mu_t -\mu)/a_t=-1$. By the second identity we explained above, we know that there is a subsequence $t_k$ with $(\widehat\mu_{t_k}-\mu)/a_{t_k}\to -1$. As a result, for all large $k$ we have that $\widehat\mu_{t_k}<\mu$ and therefore $\frac{\Delta_{t_k}}{a_{t_k}}=\frac{\mu-\widehat\mu_{t_k}}{a_{t_k}}=-\frac{\widehat\mu_{t_k}-\mu}{a_{t_k}}\to 1$. On the other hand, we know that for all $t$, $\Delta_t\le |\widehat\mu_t-\mu|$ and $\limsup |\widehat\mu_t-\mu|/a_t=1$. So, $\limsup \Delta_t/a_t\le 1$. If we now combine both we get that $\limsup \Delta_t/a_t=1$, which is exactly \eqref{eq:Delta-LIL}. Let us now derive a deterministic lower bound on $\KLinf(\widehat P_t,\mu)$ in terms of $\Delta_t$ and $\widehat\sigma_t^2$. To begin, consider a particular $t\ge 1$ and an arbitrary measure $Q\in\mathcal P([a,b])$ such that $\int xdQ(x)\ge \mu$. Let $\lambda\ge 0$ be such that $\lambda(b-\mu)<1$ so that for every $x\in[a,b]$ we have that $1+\lambda(\mu-x)\ge 1-\lambda(b-\mu)>0$. As a result of this, we know that $\varphi_\lambda(x):=\log(1+\lambda(\mu-x))$ will be well-defined indeed and bounded on $[a,b]$. Let's now apply Lemma~\ref{lem:DV-lower} with $\nu=\widehat P_t$, this particular $Q$, and $\varphi=\varphi_\lambda$. Doing this gets us
\begin{align*}
\KL(\widehat P_t\|Q) &\ge \int \log(1+\lambda(\mu-x))\,d\widehat P_t(x) -\log\left(\int (1+\lambda(\mu-x))\,dQ(x)\right).
\end{align*}
We know that the second term is controlled by the mean constraint. Hence,
\[
\int(1+\lambda(\mu-x))\,dQ(x)=1+\lambda\Big(\mu-\int x\,dQ(x)\Big)\le 1.
\]

\noindent Obviously, the function $\log(\cdot)$ is increasing, which implies that,
\[
-\log\left(\int(1+\lambda(\mu-x))\,dQ(x)\right)\ge-\log(1)=0.
\]
So! It follows that for every such $Q$ and every $\lambda\in$ $[0,1/(b-\mu))$ we have that,
\[
\KL(\widehat P_t\|Q)\ \ge\ \int \log(1+\lambda(\mu-x))\,d\widehat P_t(x).
\]

\noindent Then, taking the infimum very quickly over all feasible $Q$ gives us the lower bound,
\begin{equation}\label{eq:dual-lower-bound}
\KL_{\inf}(\widehat P_t,\mu) \ge \sup_{0\le \lambda<1/(b-\mu)} \int \log(1+\lambda(\mu-x))\,d\widehat P_t(x).
\end{equation}

\noindent The key beautiful point here is that we now have a lower bound on $\KLinf$ only in terms of the empirical measure. With this said, we will now evaluate the right-hand side here at a very carefully chosen $\lambda$ which matches the local quadratic scale we've been seeing for the KL. Namely define,
\[
\lambda_t := 
\begin{cases}
\frac{\Delta_t}{\widehat\sigma_t^2}, & \text{if }\Delta_t>0\text{ and }\widehat\sigma_t^2>0,\\
0, & \text{otherwise}.
\end{cases}
\]
Now, on this probability one event, we know that $\widehat\sigma_t^2\to\sigma^2>0$ and $\Delta_t\to 0$. As a consequence, $\lambda_t\to 0$. Meaning, for all $t$ sufficiently large we have that $\lambda_t(b-\mu)<1$ and thus $\lambda_t$ is admissible in \eqref{eq:dual-lower-bound}. Hence indeed for all such large $t$ we have that,
\begin{equation}\label{eq:eval-at-lambda-t}
\KLinf(\widehat P_t,\mu) \ge \int \log(1+\lambda_t(\mu-x))\,d\widehat P_t(x).
\end{equation}

\noindent Now, let $X\sim \widehat P_t$ and set $U_t:=\lambda_t(\mu-X)$. Here, again, as per our convention $\Ebb_{\widehat P_t}$ is integration with respect to the empirical measure. We know that $X\in[a,b]$ necessarily so we must have $|\mu-X|\le b-a$. Hence it follows that with probability one, $|U_t|\le r_t$ where $r_t:=\lambda_t(b-a)=\frac{\Delta_t}{\widehat\sigma_t^2}(b-a)\xrightarrow[t\to\infty]{}0$. Now, since $r_t\to 0$ almost surely, we know that there exists with probability one a finite random time $T_r$ such that for all $t\ge T_r$ we have that $r_t<1$. And when $r_t=0$ we have that $\lambda_t=0$ so that would mean $U_t\equiv0$ which means that the inequality is trivial. So, when we are applying Lemma~\ref{lem:Taylor}, we certainly may restrict our attention to the case where $r_t\in(0,1)$ strictly. To this end, applying Lemma~\ref{lem:Taylor} with $r=r_t$ pointwise to $u=U_t$ gives us that $\log(1+U_t)\ge U_t-\frac{U_t^2}{2}-\frac{|U_t|^3}{3(1-r_t)^3}$. Let's now take expectations with respect to $\widehat P_t$. Doing this gets us,
\begin{equation}\label{eq:log-expansion-lower}
\int \log(1+\lambda_t(\mu-x))\,d\widehat P_t(x) \ge \Ebb_{\widehat P_t}[U_t] - \frac{1}{2}\Ebb_{\widehat P_t}[U_t^2] - \frac{1}{3(1-r_t)^3}\Ebb_{\widehat P_t}[|U_t|^3].
\end{equation}

\noindent We're now going to compute each one of these terms explicitly in terms of both $\Delta_t$ and $\widehat\sigma_t^2$. For the first term we know that,
\[
\Ebb_{\widehat P_t}[U_t]=\lambda_t\,\Ebb_{\widehat P_t}[\mu-X]=\lambda_t(\mu-\widehat\mu_t)=\lambda_t\,\Delta_t.
\]

\noindent Now, let's do the second term. We know that $\Ebb_{\widehat P_t}[U_t^2]=\lambda_t^2\Ebb_{\widehat P_t}[(\mu-X)^2]$. Now, of course, $\widehat\mu_t=\Ebb_{\widehat P_t}[X]$ and thus we can easily expand $(\mu-X)^2$ around $\widehat\mu_t$. How? We can use the basic fact that $\mu-X=(\mu-\widehat\mu_t)+(\widehat\mu_t-X)$ which holds for all $t$. Let's now square and take expectations with respect to $\widehat P_t$. That is,
\begin{align*}
\Ebb_{\widehat P_t}[(\mu-X)^2]
&=\Ebb_{\widehat P_t}\big[(\mu-\widehat\mu_t)^2+2(\mu-\widehat\mu_t)(\widehat\mu_t-X) +(\widehat\mu_t-X)^2\big]\\
&=(\mu-\widehat\mu_t)^2 +2(\mu-\widehat\mu_t)\underbrace{\Ebb_{\widehat P_t}[\widehat\mu_t-X]}_{=0} +\underbrace{\Ebb_{\widehat P_t}[(\widehat\mu_t-X)^2]}_{=\widehat\sigma_t^2}\\
&=\widehat\sigma_t^2+(\mu-\widehat\mu_t)^2.
\end{align*}

\noindent Therefore, we get that $\Ebb_{\widehat P_t}[U_t^2]=\lambda_t^2\big(\widehat\sigma_t^2+(\mu-\widehat\mu_t)^2\big)$. Lastly, by definition we know that $\lambda_t=\Delta_t/\widehat\sigma_t^2$. And, we know that $\lambda_t=0$ whenever $\Delta_t=0$. As such, we have that $\lambda_t^2(\mu-\widehat\mu_t)^2=\lambda_t^2\Delta_t^2$ for all $t$. Hence we get that,
\[
\Ebb_{\widehat P_t}[U_t^2]=\lambda_t^2(\widehat\sigma_t^2+\Delta_t^2).
\]

\noindent So, that concludes the second term. It remains to compute the third final moment. Let's do it. We can simply use the fact that $|\mu-X|\le b-a$. This gives us that
\[
\Ebb_{\widehat P_t}[|U_t|^3]=\lambda_t^3\,\Ebb_{\widehat P_t}[|\mu-X|^3]\le \lambda_t^3(b-a)^3.
\]
Now we've computed all the terms. Hence let us substitute them back into \eqref{eq:log-expansion-lower}. Doing this gives us
\begin{align*}
\int \log(1+\lambda_t(\mu-x))\,d\widehat P_t(x) &\ge \lambda_t\Delta_t-\frac{1}{2}\lambda_t^2(\widehat\sigma_t^2+\Delta_t^2)-\frac{1}{3(1-r_t)^3}\lambda_t^3(b-a)^3.
\end{align*}

\noindent Let's place our attention on the relevant large $t$ regime where $\Delta_t>0$ and $\widehat\sigma_t^2>0$. In doing so, let us plug in $\lambda_t=\Delta_t/\widehat\sigma_t^2$. To this end for the first term we get that $\lambda_t\Delta_t = \frac{\Delta_t^2}{\widehat\sigma_t^2}$. The quadratic term becomes just $\frac{1}{2}\lambda_t^2(\widehat\sigma_t^2+\Delta_t^2)=\frac{1}{2}\frac{\Delta_t^2}{\widehat\sigma_t^4}(\widehat\sigma_t^2+\Delta_t^2)=\frac{\Delta_t^2}{2\widehat\sigma_t^2}+\frac{\Delta_t^4}{2\widehat\sigma_t^4}$. So overall, we know that the first two contributions are going to simplify as
\[
\lambda_t\Delta_t-\frac{1}{2}\lambda_t^2(\widehat\sigma_t^2+\Delta_t^2)=\frac{\Delta_t^2}{2\widehat\sigma_t^2}-\frac{\Delta_t^4}{2\widehat\sigma_t^4}.
\]

\noindent Lastly, we know that the cubic remainder term is bounded by $\frac{1}{3(1-r_t)^3}\lambda_t^3(b-a)^3=\frac{(b-a)^3}{3(1-r_t)^3}\cdot \frac{\Delta_t^3}{\widehat\sigma_t^6}$. Let's take all that we have shown now. Leveraging it all, we get the lower bound
\begin{equation}\label{eq:KLinf-lower-quadratic}
\int \log(1+\lambda_t(\mu-x))\,d\widehat P_t(x)\ge \frac{\Delta_t^2}{2\widehat\sigma_t^2} -\frac{\Delta_t^4}{2\widehat\sigma_t^4}-\frac{(b-a)^3}{3(1-r_t)^3}\cdot \frac{\Delta_t^3}{\widehat\sigma_t^6}.
\end{equation}

\noindent We're finally ready to go back to \eqref{eq:eval-at-lambda-t}. Indeed, we therefore are going to have for all $t$ large enough that
\[
\KL_{\inf}(\widehat P_t,\mu) \ge \frac{\Delta_t^2}{2\widehat\sigma_t^2}- \frac{\Delta_t^4}{2\widehat\sigma_t^4}- \frac{(b-a)^3}{3(1-r_t)^3}\cdot \frac{\Delta_t^3}{\widehat\sigma_t^6}.
\]

\noindent Now, $\Delta_t\to 0$. And, $\widehat\sigma_t^2\to\sigma^2>0$. Also, $r_t\to 0$ almost surely. As a consequence of all this, we know that both of the error terms are $o\big(\Delta_t^2/\widehat\sigma_t^2\big)$ almost surely. To be concrete we have that with probability one $\frac{\Delta_t^4/\widehat\sigma_t^4}{\Delta_t^2/\widehat\sigma_t^2}=\frac{\Delta_t^2}{\widehat\sigma_t^2}\xrightarrow[t\to\infty]{}0$, $\frac{\Delta_t^3/\widehat\sigma_t^6}{\Delta_t^2/\widehat\sigma_t^2}=\frac{\Delta_t}{\widehat\sigma_t^4}\xrightarrow[t\to\infty]{}0$ and that $(1-r_t)^{-3}\to 1$. As a consequence we get this beautiful local quadratic lower bound almost surely as $t\to\infty$,
\begin{equation}\label{eq:KL-inf-asymp}
\KLinf(\widehat P_t,\mu) \ge (1-o(1))\,\frac{\Delta_t^2}{2\widehat\sigma_t^2}.
\end{equation}

\noindent Now, we are going to combine \eqref{eq:KL-inf-asymp} with the mean LIL information that we got, \eqref{eq:Delta-LIL}. Of course, since $\widehat\sigma_t^2\to\sigma^2$, we get that,
\[
\limsup_{t\to\infty}\frac{t}{\log\log t}\cdot \frac{\Delta_t^2}{2\widehat\sigma_t^2} =\limsup_{t\to\infty}\left(\frac{\sigma^2}{\widehat\sigma_t^2}\right)\left(\frac{\Delta_t^2}{(2\sigma^2\log\log t)/t}\right)=1,
\]
almost surely. Now, let us take a particular but arbitrary $\varepsilon\in(0,1)$. Certainly we know by $\eqref{eq:KL-inf-asymp}$ that on our almost sure event there exists a random $T_\varepsilon<\infty$ such that for all $t\ge T_\varepsilon$,
\[
\KLinf(\widehat P_t,\mu)\ge (1-\varepsilon)\frac{\Delta_t^2}{2\widehat\sigma_t^2}.
\]

\noindent Let's now multiply by $t/\log\log t$ and then take the $\limsup$ over $t\to\infty$. Doing so gives us that,
\[
\limsup_{t\to\infty}\frac{t\,\KL_{\inf}(\widehat P_t,\mu)}{\log\log t}\ge (1-\varepsilon)\limsup_{t\to\infty}\frac{t}{\log\log t}\cdot\frac{\Delta_t^2}{2\widehat\sigma_t^2} =(1-\varepsilon)\cdot 1.
\]

\noindent Note that $\varepsilon$ was of course arbitrary. That means that almost surely,
\[
\limsup_{t\to\infty}\frac{t\,\KLinf(\widehat P_t,\mu)}{\log\log t}\ge 1.
\]
And taken together with our Theorem~\ref{thm:LIL-KLinf}'s already-proven upper bound, we indeed have equality. That is,
\[
\limsup_{t\to\infty}\frac{t\,\KLinf(\widehat P_t,\mu)}{\log\log t}=1.
\]
And, with that, we are done at last.
\end{proof}

\section{Beyond Bounded Case: An Extension to Sub-Gaussian and other such tail regimes}\label{sec:beyond-bounded}
%\textcolor{red}{Ashwin: to be done.}
The bounded support condition that we imposed is in fact quite important, as we already detailed. The reason is simple: we are able to keep the $\KLinf$ controlled, particularly if mass is placed at arbitrary far out points on the right tail. Therefore, in this section, we will start by showing this degeneracy of the $\KLinf$ formally. We will then provide a general theorem to recover the very asymptotic tail patterns in multiple regimes (eg, subgaussian tails). Intuitively, what we do is enforce an natural (and of course, apt) constraint on the competing model class. Having said all of that, the very first thing we are going to show rigorously this idea of the $\KLinf$ collapsing on $\mathbb R$ when it is not constrained. Let $\nu$ be any probability measure on $\mathbb R$ with finite mean $\bar \mu:= \int x\,d\nu(x)$. Let us consider the unconstrained definition,
\[
\KL_{\inf}^{\mathbb R}(\nu,m) :=\inf\Big\{\KL(\nu\|Q): Q\in\mathcal P(\mathbb R),\ \int x\,dQ(x)\ge m\Big\}.
\]
We show degeneracy of it without a tail constraint in the next proposition.
\begin{proposition}\label{prop:KLinf-degenerate}
Suppose $m>\bar\mu$. Then, $\KLinf^{\mathbb R}(\nu,m)=0$.
\end{proposition}
\noindent What this particular proposition is telling us is precisely the consequence of being allowed to ``sprinkle'' the mass far out. That is, it's easy to pay a very small (in fact, arbitrarily small) KL cost while still achieving an increase in the mean. Why? The class of competitors is not constrained. So, a natural question that we all would have upon seeing this issue is how do we escape this issue? In other words, can we generalize beyond bounded random variables? The simple and natural answer is that we need only impose a constraint to eliminate such ``sprinkling.'' An easy expansion is if we consider the class of tail-controlled variables.\footnote{These include, for instance sub-Gaussian, sub-exponential, finite moments, etc.} For this class, let us leverage a growing envelope that will almost surely eventually contain all samples. Gorgeously, this will yield us a time constrained and nontrivial $\KLinf$. To that end, our previous analysis with random variables with bounded support can easily be extended. That will thus be the crux of this section. To begin, let's take $(B_t)_{t\ge 1}$ to be a deterministic and nondecreasing sequence whereby $B_t\to\infty$. Define then the time-$t$ constrained version as,
\[
\KL_{\inf}^{(t)}(\nu, m):=\inf\Big\{\KL(\nu\|Q): Q\in\mathcal P([-B_t, B_t]),\ \int xdQ(x)\ge m\Big\}.
\]

\noindent For fun, suppose that $\nu=\widehat P_t$. In this case, the interpretation is as follows: what is the ``budget'' with respect to the KL that we would need to perturb the empirical law into another distribution taking support on $[-B_t, B_t]$ with mean at least $m$. With all that said, we will now present our main theorem result. To prelude, it is intuitively telling us that if $B_t$ grows slowly enough compared to the LIL scale, then we again will get the sharp asymptotic LIL constant $1$ indeed. In short we present a theorem that gives a two sided LIL under a growing envelope.

\begin{theorem}\label{thm:LIL-growing-envelope}
Suppose that $X_1, X_2, X_3, \dots$ are iid real-valued random variables with mean $\mu$ and variance $\sigma^2\in(0,\infty)$. Assume that almost surely there exists a finite (possibly random) $T_B$ such that $|X_i|\le B_t$ for all $t\ge T_B$ and all $1\le i\le t$,
for some deterministic nondecreasing sequence $(B_t)_{t\ge1}$ such that as $t\to\infty$
\[
B_t = o\left(\sqrt{\frac{t}{\log\log t}}\right).
\]
 Then,
\[
\boxed{
\bbP\left(\limsup_{t\to\infty}\frac{t\,\KLinf^{(t)}(\widehat P_t,\mu)}{\log\log t}=1\right)=1.
}
\]
\end{theorem}

\noindent The proof of this theorem follows closely that of the bounded case, with minor adjustments. Hence we defer it to the appendix along with the lemmas and proposition proofs. But, remarkably, the takeaway is as follows. If we can control the tails in a way to get us the almost sure envelope $B_t=o(\sqrt{t/\log\log t})$, then our bounded support $\KLinf$ asymptotic LIL argument extends. Meaning, after verifying that the uniform smallness parameter $r_t:=2B_t\Delta_t/\widehat\sigma_t^2$ satisfies $r_t\to 0$ wp 1, the LIL argument extends to the time-varying constraint $\KLinf^{(t)}$. And in fact the $\limsup$ constant we get remains the sharp $1$. Note that in Appendix~\ref{sec:further-applications-beyond-bounded} we give three broad instantiations of Theorem~\ref{thm:LIL-growing-envelope} holding: sub-Gaussian, subexponential, and finite $p>2$ moments. Finally, in Appendix~\ref{sec:env-tightness} we show tightness of the envelope and why our above Theorem excludes the case where $p=2$. Given the importance of Appendix~\ref{sec:env-tightness}, we explain the main ideas as follows. First, Lemma~\ref{lem:sprinkling-upper} gives a feasible competitor $Q$, which we get from sprinkling a small mass $\varepsilon$ at the boundary point $B$. From this, we get the upper bound $\KLinf^{(t)}(\nu, m)\le \KL(\nu\|Q)\le -\log(1-\varepsilon)$, and equality in fact when $\nu(\{B\})=0$, with $\varepsilon=(m-\bar\mu)/(B-\bar\mu)$ and $-\log(1-\varepsilon)\sim \varepsilon$ as $\varepsilon\downarrow0$. This is taken further with Proposition~\ref{prop:large-envelope-collapse} where this so-called sprinkling bound is combined with the mean LIL. Specifically, through this proposition we show that whenever an almost sure valid deterministic envelope satisfies $B_t\sqrt{t/\log\log t}\to\infty$, then the constrained empirical cost will collapse on the $\log\log t$ scale. Dangerously, we would have that,
\[
\limsup_{t\to\infty}\frac{t\,\KLinf^{(t)}(\widehat P_t,\mu)}{\log\log t}=0.
\]
Then, in Lemma~\ref{lem:envelope-p2} we give a Borel Cantelli criterion under the weak second-moment assumption to show (first) validity of a particular envelope. In other words, if indeed $\sum_t 1/B_t^2<\infty$, then we have that $|X_i|\le B_t$ for all $i\le t$ eventually almost surely. That is: $B_t=\max_{s\le t}\sqrt{s}(\log s)^\gamma$ with $\gamma>\tfrac12$ is an almost sure valid deterministic envelope. We take this idea and combine it with Proposition~\ref{prop:large-envelope-collapse} in our Corollary~\ref{cor:p2-envelope-limsup0}. In particular, under $\Ebb[X_1^2]<\infty$ such envelopes would force the same collapse of the $\log\log t$ normalization. We conclude the section with Proposition~\ref{prop:p2-no-small-envelope} where we show that in general we cannot verify the key assumption of Theorem~\ref{thm:LIL-growing-envelope} under only $p=2$: a probability one envelope with $B_t=o(\sqrt{t/\log\log t})$. Note that we do this by building a distribution with finite variance for which $|X_t|>\sqrt{t/\log\log t}$ infinitely often almost surely.

Together, all of the results in Appendix~\ref{sec:env-tightness} tell us that $\sqrt{t/\log\log t}$ is a sharp boundary for deterministic envelopes. Meaning: if we go below the boundary (assuming that the envelope can be verified), we will remain in the local quadratic regime of Theorem~\ref{thm:LIL-growing-envelope} with the sharp constant $1$. The sharpness breaks, however, if we go above the envelope when sprinkling dominates and forces the normalized cost to $0$.

\section{Conclusion}
In this paper, we established the first exact LIL for the empirical $\KLinf$. Namely we showed that
\[
\limsup_{t\to\infty}\frac{t\,\KLinf(\what P_t,\mu)}{\log\log t}=1\quad\text{almost surely.}
\]
We did this first for random variables who take values on a compact support. We then extend it to unbounded data via slowly growing envelopes $[-B_t, B_t]$ with $B_t=o(\sqrt{t/\log\log t})$ that will eventually contain the data almost surely. Crucially, we have the local equivalence with probability once that $\KLinf(\hat P_t, \mu)=(1+o(1))(\mu-\what\mu_t)_+^2/(2\what\sigma_t^2)$. The benefit of this local equivalence is that it makes the iterated-logarithm constant immediate and obvious from the classical mean LIL. With all this being said, certainly, there are several promising areas of future work. An immediate one is when analyzing higher-dimensional (or multiple-moment) constraints. In these situations, curvature and feasibility become much more subtle, albeit information projections are still quite natural. In addition, our paper worked under the standard i.i.d.\ data setup. As such, it would certainly be valuable to extend these results to dependent data (eg: martingales, mixing, or Markov chains), where certainly the LIL behavior could persist. But different tools would be needed to analyze them \citep{stout1970, delaPena2009, howard2021, Dembo1998}.

\bibliography{references}

@article{khinchin1924,
  author  = {Khinchin, A.},
  title   = {{\"U}ber einen Satz der Wahrscheinlichkeitsrechnung},
  journal = {Fundamenta Mathematicae},
  volume  = {6},
  number = {1},
  pages   = {9--20},
  year    = {1924},
  url =     {http://eudml.org/doc/214283}
}

@article{Kolmogoroff1929,
  author  = {Kolmogoroff, A.},
  title   = {{\"U}ber das Gesetz des iterierten Logarithmus},
  journal = {Mathematische Annalen},
  year    = {1929},
  volume  = {101},
  number  = {1},
  pages   = {126--135},
  doi     = {10.1007/BF01454828},
  url     = {http://eudml.org/doc/159322}
}

@article{Feller1943,
  author  = {Feller, Will},
  title   = {The General Form of the So-Called Law of the Iterated Logarithm},
  journal = {Transactions of the American Mathematical Society},
  year    = {1943},
  volume  = {54},
  number  = {3},
  pages   = {373--402},
  doi     = {10.1090/S0002-9947-1943-0009263-7}
}

@article{Chung1948,
  author  = {Chung, Kai Lai},
  title   = {On the Maximum Partial Sums of Sequences of Independent Random Variables},
  journal = {Transactions of the American Mathematical Society},
  year    = {1948},
  volume  = {64},
  number  = {2},
  pages   = {205--233},
  doi     = {10.1090/S0002-9947-1948-0026274-0},
  url     = {https://www.ams.org/journals/tran/1948-064-02/S0002-9947-1948-0026274-0/}
}

@article{Strassen1964,
  author  = {Strassen, Volker},
  title   = {An Invariance Principle for the Law of the Iterated Logarithm},
  journal = {Zeitschrift f{\"u}r Wahrscheinlichkeitstheorie und Verwandte Gebiete},
  year    = {1964},
  volume  = {3},
  number  = {3},
  pages   = {211--226},
  doi     = {10.1007/BF00534910}
}

@article{deacosta1983,
  author  = {de Acosta, Alejandro},
  title   = {A New Proof of the {Hartman--Wintner} Law of the Iterated Logarithm},
  journal = {The Annals of Probability},
  year    = {1983},
  volume  = {11},
  number  = {2},
  pages   = {270--276},
  doi     = {10.1214/aop/1176993596}
}

@article{bingham1986,
  author  = {Bingham, Nicholas H.},
  title   = {Variants on the Law of the Iterated Logarithm},
  journal = {Bulletin of the London Mathematical Society},
  year    = {1986},
  volume  = {18},
  number  = {5},
  pages   = {433--467},
  doi     = {10.1112/blms/18.5.433}
}

@article{stout1970,
  author  = {Stout, William F.},
  title   = {The {Hartman--Wintner} Law of the Iterated Logarithm for Martingales},
  journal = {The Annals of Mathematical Statistics},
  year    = {1970},
  volume  = {41},
  number  = {6},
  pages   = {2158--2160},
  doi     = {10.1214/aoms/1177696716}
}

@book{stout1974,
  author    = {Stout, William F.},
  title     = {Almost Sure Convergence},
  publisher = {Academic Press},
  year      = {1974},
  isbn      = {978-0126727500}
}

@article{darlingrobbins1967,
  author  = {Darling, Donald A. and Robbins, Herbert},
  title   = {Confidence Sequences for Mean, Variance, and Median},
  journal = {Proceedings of the National Academy of Sciences},
  year    = {1967},
  volume  = {58},
  number  = {1},
  pages   = {66--68},
  doi     = {10.1073/pnas.58.1.66}
}

@misc{balsubramani2015,
      title={Sharp Finite-Time Iterated-Logarithm Martingale Concentration}, 
      author={Akshay Balsubramani},
      year={2015},
      eprint={1405.2639},
      archivePrefix={arXiv},
      primaryClass={math.PR},
      url={https://arxiv.org/abs/1405.2639}, 
}

@article{Howard2020,
  author  = {Howard, Steven R. and Ramdas, Aaditya and McAuliffe, Jon and Sekhon, Jasjeet},
  title   = {Time-uniform {Chernoff} Bounds via Nonnegative Supermartingales},
  journal = {Probability Surveys},
  year    = {2020},
  volume  = {17},
  pages   = {257--317},
  doi     = {10.1214/18-PS321}
}

@inproceedings{kaufmann2012,
  author    = {Kaufmann, Emilie and Capp{\'e}, Olivier and Garivier, Aur{\'e}lien},
  title     = {On {Bayesian} Upper Confidence Bounds for Bandit Problems},
  booktitle = {Proceedings of the Fifteenth International Conference on Artificial Intelligence and Statistics (AISTATS)},
  series    = {Proceedings of Machine Learning Research},
  volume    = {22},
  pages     = {592--600},
  year      = {2012},
  publisher = {PMLR}
}

@book{Gittins1989,
  author    = {Gittins, John C.},
  title     = {Multi-armed Bandit Allocation Indices},
  series    = {Wiley-Interscience Series in Systems and Optimization},
  publisher = {John Wiley \& Sons},
  year      = {1989},
  isbn      = {978-0471920595}
}

@article{Robbins1952,
  author  = {Robbins, Herbert},
  title   = {Some Aspects of the Sequential Design of Experiments},
  journal = {Bulletin of the American Mathematical Society},
  year    = {1952},
  volume  = {58},
  number  = {5},
  pages   = {527--535},
  doi     = {10.1090/S0002-9904-1952-09620-8}
}

@article{Teicher1974,
  author  = {Teicher, Henry},
  title   = {On the Law of the Iterated Logarithm},
  journal = {The Annals of Probability},
  year    = {1974},
  volume  = {2},
  number  = {4},
  pages   = {714--728},
  doi     = {10.1214/aop/1176996614}
}

@book{delaPena2009,
  author    = {de la Pe{\~n}a, Victor H. and Lai, Tze Leung and Shao, Qi-Man},
  title     = {Self-Normalized Processes: Limit Theory and Statistical Applications},
  series    = {Probability and Its Applications},
  publisher = {Springer},
  year      = {2009},
  doi       = {10.1007/978-3-540-85636-8},
  isbn      = {978-3-540-85635-1}
}

@book{Lattimore2020,
  author    = {Lattimore, Tor and Szepesv{\'a}ri, Csaba},
  title     = {Bandit Algorithms},
  publisher = {Cambridge University Press},
  year      = {2020},
  doi       = {10.1017/9781108571401},
  isbn      = {978-1108486828}
}

@inproceedings{hondatakemura2010,
  author    = {Honda, Junya and Takemura, Akimichi},
  title     = {An Asymptotically Optimal Bandit Algorithm for Bounded Support Models},
  booktitle = {Proceedings of the 23rd Conference on Learning Theory (COLT)},
  pages     = {67--79},
  year      = {2010}
}

@article{cappe2013,
  author  = {Capp{\'e}, Olivier and Garivier, Aur{\'e}lien and Maillard, Odalric-Ambrym and Munos, R{\'e}mi and Stoltz, Gilles},
  title   = {{Kullback--Leibler} Upper Confidence Bounds for Optimal Sequential Allocation},
  journal = {The Annals of Statistics},
  year    = {1924},
  volume  = {41},
  number  = {3},
  pages   = {1516--1541},
  doi     = {10.1214/13-AOS1119}
}

@article{hondatakemura2015,
  author  = {Honda, Junya and Takemura, Akimichi},
  title   = {Non-asymptotic Analysis of a New Bandit Algorithm for Semi-bounded Rewards},
  journal = {Journal of Machine Learning Research},
  volume  = {16},
  number  = {1},
  pages   = {3721--3756},
  year    = {2015}
}

@inproceedings{gariviercappe2011,
  author    = {Garivier, Aur{\'e}lien and Capp{\'e}, Olivier},
  title     = {The {KL-UCB} Algorithm for Bounded Stochastic Bandits and Beyond},
  booktitle = {Proceedings of the 24th Annual Conference on Learning Theory (COLT)},
  series    = {Proceedings of Machine Learning Research},
  volume    = {19},
  pages     = {359--376},
  year      = {2011},
  publisher = {PMLR}
}

@article{Bubeck2012,
  author  = {Bubeck, S{\'e}bastien and Cesa-Bianchi, Nicol{\`o}},
  title   = {Regret Analysis of Stochastic and Nonstochastic Multi-armed Bandit Problems},
  journal = {Foundations and Trends{\textregistered} in Machine Learning},
  year    = {2012},
  volume  = {5},
  number  = {1},
  pages   = {1--122},
  doi     = {10.1561/2200000024}
}

@article{Auer2002,
  author  = {Auer, Peter and Cesa-Bianchi, Nicol{\`o} and Fischer, Paul},
  title   = {Finite-time Analysis of the Multiarmed Bandit Problem},
  journal = {Machine Learning},
  year    = {2002},
  volume  = {47},
  number  = {2},
  pages   = {235--256},
  doi     = {10.1023/A:1013689704352}
}

@article{Burnetas1996,
  author  = {Burnetas, Apostolos N. and Katehakis, Michael N.},
  title   = {Optimal Adaptive Policies for Sequential Allocation Problems},
  journal = {Advances in Applied Mathematics},
  year    = {1996},
  volume  = {17},
  number  = {2},
  pages   = {122--142},
  doi     = {10.1006/aama.1996.0007}
}

@article{Lai1985,
  author  = {Lai, Tze Leung and Robbins, Herbert},
  title   = {Asymptotically Efficient Adaptive Allocation Rules},
  journal = {Advances in Applied Mathematics},
  year    = {1985},
  volume  = {6},
  number  = {1},
  pages   = {4--22},
  doi     = {10.1016/0196-8858(85)90002-8}
}

@book{Owen2001,
  author    = {Owen, Art B.},
  title     = {Empirical Likelihood},
  series    = {Monographs on Statistics and Applied Probability},
  volume    = {92},
  publisher = {Chapman \& Hall/CRC},
  year      = {2001},
  doi       = {10.1201/9781420036152},
  isbn      = {978-1584880714}
}

@article{Owen1988,
  author  = {Owen, Art B.},
  title   = {Empirical Likelihood Ratio Confidence Intervals for a Single Functional},
  journal = {Biometrika},
  year    = {1988},
  volume  = {75},
  number  = {2},
  pages   = {237--249},
  doi     = {10.1093/biomet/75.2.237}
}

@book{Boucheron2013,
  author    = {Boucheron, St{\'e}phane and Lugosi, G{\'a}bor and Massart, Pascal},
  title     = {Concentration Inequalities: A Nonasymptotic Theory of Independence},
  publisher = {Oxford University Press},
  year      = {2013},
  doi       = {10.1093/acprof:oso/9780199535255.001.0001},
  isbn      = {978-0199535255}
}

@book{Dupuis1997,
  author    = {Dupuis, Paul and Ellis, Richard S.},
  title     = {A Weak Convergence Approach to the Theory of Large Deviations},
  series    = {Wiley Series in Probability and Statistics},
  publisher = {Wiley-Interscience},
  year      = {1997},
  doi       = {10.1002/9781118165904},
  isbn      = {978-0-471-07672-8}
}

@book{Dembo1998,
  author    = {Dembo, Amir and Zeitouni, Ofer},
  title     = {Large Deviations Techniques and Applications},
  series    = {Stochastic Modelling and Applied Probability},
  volume    = {38},
  edition   = {2nd},
  publisher = {Springer},
  year      = {1998},
}

@article{Sanov1957,
  author  = {Sanov, {Igor Nikolaevich}},
  title   = {On the Probability of Large Deviations of Random Magnitudes},
  journal = {Matematicheskii Sbornik (N.S.)},
  year    = {1957},
  volume  = {42(84)},
  number  = {1},
  pages   = {11--44}
}

@article{Csiszar2003,
  author  = {Csisz{\'a}r, Imre and Mat{\'u}{\v{s}}, Franti{\v{s}}ek},
  title   = {Information Projections Revisited},
  journal = {IEEE Transactions on Information Theory},
  year    = {2003},
  volume  = {49},
  number  = {6},
  pages   = {1474--1490},
  doi     = {10.1109/TIT.2003.810633}
}

@article{Csiszar1984,
  author  = {Csisz{\'a}r, Imre},
  title   = {Sanov Property, Generalized {$I$}-Projection and a Conditional Limit Theorem},
  journal = {The Annals of Probability},
  year    = {1984},
  volume  = {12},
  number  = {3},
  pages   = {768--793},
  doi     = {10.1214/aop/1176993227}
}

@article{Csiszar1975,
  author  = {Csisz{\'a}r, Imre},
  title   = {{$I$}-Divergence Geometry of Probability Distributions and Minimization Problems},
  journal = {The Annals of Probability},
  year    = {1975},
  volume  = {3},
  number  = {1},
  pages   = {146--158},
  doi     = {10.1214/aop/1176996454}
}

@article{Cramer1938,
  author  = {Cram{\'e}r, Harald},
  title   = {Sur un nouveau th{\'e}or{\`e}me-limite de la th{\'e}orie des probabilit{\'e}s},
  journal = {Actualit{\'e}s Scientifiques et Industrielles},
  year    = {1938},
  volume  = {736},
  pages   = {5--23},
  publisher = {Hermann {\&} Cie}
}

@article{Kullback1951,
  author  = {Kullback, Solomon and Leibler, Richard A.},
  title   = {On Information and Sufficiency},
  journal = {The Annals of Mathematical Statistics},
  year    = {1951},
  volume  = {22},
  number  = {1},
  pages   = {79--86},
  doi     = {10.1214/aoms/1177729694}
}

@article{Einmahl2004,
  author  = {Einmahl, John H. J. and Rosalsky, Andrew},
  title   = {The Functional Law of the Iterated Logarithm for the Empirical Process Based on Sample Means},
  journal = {Journal of Theoretical Probability},
  year    = {2001},
  volume  = {14},
  number  = {2},
  pages   = {577--597},
  doi     = {10.1023/A:1011128101094}
}

@article{Mason2004,
  author  = {Mason, David M.},
  title   = {A Uniform Functional Law of the Logarithm for the Local Empirical Process},
  journal = {The Annals of Probability},
  year    = {2004},
  volume  = {32},
  number  = {2},
  pages   = {1391--1418},
  doi     = {10.1214/009117904000000243}
}

@article{Deheuvels1994,
  author  = {Deheuvels, Paul and Mason, David M.},
  title   = {Functional Laws of the Iterated Logarithm for Local Empirical Processes Indexed by Sets},
  journal = {The Annals of Probability},
  year    = {1994},
  volume  = {22},
  number  = {3},
  pages   = {1619--1661},
  doi     = {10.1214/aop/1176988617}
}

@book{vanderVaart1996,
  author    = {van der Vaart, Aad W. and Wellner, Jon A.},
  title     = {Weak Convergence and Empirical Processes: With Applications to Statistics},
  series    = {Springer Series in Statistics},
  publisher = {Springer},
  year      = {1996},
  doi       = {10.1007/978-1-4757-2545-2},
  isbn      = {978-0-387-94640-5}
}

@book{Shorack1986,
  author    = {Shorack, Galen R. and Wellner, Jon A.},
  title     = {Empirical Processes with Applications to Statistics},
  series    = {Wiley Series in Probability and Mathematical Statistics},
  publisher = {John Wiley \& Sons},
  year      = {1986},
}

@article{howard2021,
  author  = {Steven R. Howard and Aaditya Ramdas and Jon McAuliffe and Jasjeet S. Sekhon},
  title   = {Time-uniform, nonparametric, nonasymptotic confidence sequences},
  journal = {Annals of Statistics},
  year    = {2021},
  volume  = {49},
  number  = {2},
  pages   = {1055--1080},
  doi     = {10.1214/20-AOS1991}
}

@article{DonskerVaradhan1975,
  author  = {Monroe D. Donsker and S. R. Srinivasa Varadhan},
  title   = {Asymptotic Evaluation of Certain Markov Process Expectations for Large Time, I},
  journal = {Communications on Pure and Applied Mathematics},
  year    = {1975},
  volume  = {28},
  number  = {1},
  pages   = {1--47},
  doi     = {10.1002/cpa.3160280102}
}

@book{wald1947,
  author = {Wald, Abraham},
  title  = {Sequential Analysis},
  publisher = {John Wiley \& Sons},
  year = {1947}
}

@article{hartman1941law,
  title={On the law of the iterated logarithm},
  author={Hartman, Philip and Wintner, Aurel},
  journal={American Journal of Mathematics},
  volume={63},
  number={1},
  pages={169--176},
  year={1941},
  publisher={Johns Hopkins University Press}
}
\appendix
\section{Omitted Proofs}\label{sec:full-proofs}
\begin{proof}[Proof of Lemma~\ref{lem:Taylor}]
This is a very easy proof to see. Let's first fix $r\in(0,1)$ and of course $u\in[-r, r]$. We will use Taylor's theorem with a remainder to approximate $f(u)$ around $u=0$. We know that $f(u)=-\log(1+u)$ is thrice differentiable (ie, in $C^3$) on the interval $[-r, r]$. Therefore, we can write the third order expansion as, 
\[
f(u) = f(0)+f'(0)u+\frac{f''(0)}{2}u^2+\frac{f^{(3)}(\xi)}{6}u^3,
\]
where $\xi$ is some intermediate value that's between $0$ and $u$. Let us now calculate the first three derivatives of $f(u)$ and evaluate them at $u=0$. Clearly, $f(0)=-\log(1+0)=0.$ Then, $f'(u)=-\frac{1}{1+u}$, so $f'(0)=-1$, and $f''(u)=\frac{1}{(1+u)^2}$, so $f''(0)=1$. Finally, $f^{(3)}(v)=-\frac{2}{(1+v)^3}$. Obviously, since $\xi\in[-r, r]$, we have that $(1+\xi)\ge (1-r)$. Hence we can bound the remainder term easily to not depend on the exact value of $\xi$: $|f^{(3)}(\xi)| \le \frac{2}{(1-r)^3}$. Therefore, we get that,
\[
f(u) \le -u + \frac{u^2}{2} + \frac{1}{6}\cdot \frac{2}{(1-r)^3}\,|u|^3 = -u + \frac{u^2}{2} + \frac{|u|^3}{3(1-r)^3},
\]
which was exactly our claim and hence we are done!
\end{proof}

\begin{proof}[Proof of Lemma~\ref{lem:DV-lower}]
Suppose $\int e^{\varphi}dQ=+\infty$. If this indeed were the case, then we'd know that the right hand side in our lemma becomes $-\infty$. As such, the inequality would hold trivially. Therefore, wlog assume that $\int e^{\varphi}dQ\in(0,\infty)$ and set $Z:=\int e^{\varphi}dQ$. Then let us define a probability measure $Q^{\varphi}$ by exponential tilting. That is, $\frac{dQ^{\varphi}}{dQ}:=\frac{e^{\varphi}}{Z}$. We know that $\varphi$ is real-valued, so it follows that $e^{\varphi}>0$ everywhere, so it follows that $Q^{\varphi}$ and $Q$ have the same null sets. Specifically, $\nu\ll Q$ implies that $\nu\ll Q^{\varphi}$. Looking at the set where $\frac{d\nu}{dQ^{\varphi}}$ is defined which is $\nu$ almost surely, let us write the Radon-Nikodym chain rule as, $\frac{d\nu}{dQ}=\frac{d\nu}{dQ^{\varphi}}\cdot\frac{dQ^{\varphi}}{dQ}$. Therefore, we get that $\log\left(\frac{d\nu}{dQ}\right)=\log\left(\frac{d\nu}{dQ^{\varphi}}\right)+\log\left(\frac{dQ^{\varphi}}{dQ}\right) = \log\left(\frac{d\nu}{dQ^{\varphi}}\right)+\varphi-\log Z$. Let us now integrate both sides with respect to $\nu$. If we do this we get that,
\begin{align*}
\KL(\nu\|Q)
&= \int \log\left(\frac{d\nu}{dQ}\right)\,d\nu \\
&= \int \log\left(\frac{d\nu}{dQ^{\varphi}}\right)\,d\nu+\int \varphi\,d\nu-\log Z \\
&= \KL(\nu\|Q^{\varphi})+\int \varphi\,d\nu-\log\left(\int e^{\varphi}\,dQ\right).
\end{align*}

\noindent Clearly, $\KL(\nu\|Q^{\varphi})\ge 0$, and it may be $+\infty$ also. But in any case, using this fact, we get that,
\[
\KL(\nu\|Q)\ge \int \varphi d\nu -\log\left(\int e^{\varphi}dQ\right),
\]
which was to be shown and hence that concludes the proof.
\end{proof}

\begin{comment}
\begin{proof}[Proof of Lemma~\ref{lem:taylor-lower-log}]
This is a very basic proof to see. Take a particular $r\in(0,1)$ and $u\in[-r,r]$. We know that the function $g(u)=\log(1+u)$ is thrice differentiable on $[-r,r]$, so Taylor's theorem with a remainder at $0$ gives us,
\[
g(u)=g(0)+g'(0)u+\frac{g''(0)}{2}u^2 + \frac{g^{(3)}(\xi)}{6}u^3,
\]
for some $\xi$ between $0$ and $u$ of course. Let's now compute these derivatives. Clearly, we know that $g(0)=0$. Second, $g'(u)=\frac{1}{1+u}$, which means that $g'(0)=1$. Third, $g''(u)=-\frac{1}{(1+u)^2}$, so $g''(0)=-1$. Fourth, $g^{(3)}(v)=\frac{2}{(1+v)^3}$. Recall that $\xi\in[-r,r]$, so it follows that $1+\xi\ge 1-r$. Therefore, $0<g^{(3)}(\xi)=\frac{2}{(1+\xi)^3}\le \frac{2}{(1-r)^3}$. Therefore, the remainder must satisfy $\left|\frac{g^{(3)}(\xi)}{6}u^3\right|\le \frac{1}{6}\cdot \frac{2}{(1-r)^3}|u|^3=\frac{|u|^3}{3(1-r)^3}$. Let's now plug this back into our Taylor formula to get that,
\[
g(u)=u-\frac{u^2}{2}+\frac{g^{(3)}(\xi)}{6}u^3 \ge u-\frac{u^2}{2}-\frac{|u|^3}{3(1-r)^3},
\]
which was exactly our claim and thus we are done.
\end{proof}
\end{comment}

\begin{proof}[Proof of Proposition~\ref{prop:KLinf-degenerate}]
We already gave intuition for this proof earlier in our first section, hence it's going to be a very obvious proof. To begin, take a $\varepsilon\in(0,1)$ and pick $M\in\mathbb R$ large enough so that we can have $(1-\varepsilon)\bar \mu+\varepsilon M\ge m$. Then, define the mixture $Q_{\varepsilon,M}:=(1-\varepsilon)\nu+\varepsilon\delta_M$. Two things then follow. First, we know that $Q_{\varepsilon,M}$ must be a probability measure on $\mathbb R$. Second, it also necessarily satisfies the mean constraint. It's easy to see why. By definition it satisfies that $\int x dQ_{\varepsilon,M}(x)=(1-\varepsilon)\int xd\nu(x)+\varepsilon M=(1-\varepsilon)\bar\mu +\varepsilon M\ge m$ by our construction of $M$. Also, for any measurable set $A$, we have the relation that $Q_{\varepsilon, M}(A)\ge (1-\varepsilon)\nu(A)$, so it must follow that $Q_{\varepsilon, M}$ dominates $\nu$. This is particularly important for us to note here because we can conclude that if $Q_{\varepsilon,M}(A)$ were to be $0$, $\nu(A)=0$ follows. So, $\nu\ll Q_{\varepsilon,M}$ and $\KL(\nu\|Q_{\varepsilon,M})$ is finite. We now will upper bound $\KL(\nu\|Q_{\varepsilon,M})$. Now, let $L:=\frac{d\nu}{dQ_{\varepsilon,M}}$. We will claim that $\nu$-almost-surely, we have that $L\le \frac{1}{1-\varepsilon}$. It's very easy to see why. Indeed, take an arbitrary $\delta>0$ and let $A_{\delta}:=\{L>\frac{1}{1-\varepsilon}+\delta\}$. Then it follows that,
\[
\nu(A_{\delta})=\int_{A_{\delta}} L\,dQ_{\varepsilon,M}>\Big(\frac{1}{1-\varepsilon}+\delta\Big)Q_{\varepsilon,M}(A_{\delta}) \ge \Big(\frac{1}{1-\varepsilon}+\delta\Big)(1-\varepsilon)\nu(A_{\delta}) =\big(1+\delta(1-\varepsilon)\big)\nu(A_{\delta}).
\]
From here, we can see that $\nu(A_{\delta})$ is forced to be $0$ in this case (if it wasn't then that would contradict the fact that it's a probability measure and nonnegative). Now, we know that $\delta>0$ was arbitrary. Hence, given that $\nu(A_\delta)$ is a measure $0$ event, it follows that almost surely under $\nu$, indeed $L\le \frac{1}{1-\varepsilon}$. Therefore, we can conclude that,
\[
\KL(\nu\|Q_{\varepsilon,M}) =\int \log\left(\frac{d\nu}{dQ_{\varepsilon,M}}\right)\,d\nu =\int \log(L)\,d\nu \le \log\left(\frac{1}{1-\varepsilon}\right) =-\log(1-\varepsilon).
\]

\noindent Note that we also chose $\varepsilon\in(0,1)$ arbitrarily. And, $-\log(1-\varepsilon)\downarrow 0$ as $\varepsilon\downarrow 0$. Hence we get that $0\le \KLinf^{\mathbb R}(\nu, m)\le \inf_{\varepsilon\in(0,1)}(-\log(1-\varepsilon))=0$. And, thus, $\KLinf^{\mathbb R}(\nu, m)$ is $0$ and so we are done.\end{proof}

\begin{proof}[Proof of Theorem~\ref{thm:LIL-growing-envelope}]
Before we begin this proof, we will clarify that throughout, we are working on the probability one event on which three things hold: the mean LIL, the fact that $\widehat\sigma_t^2\to\sigma^2$, and that the envelope property holds for all large $t$. Now to begin, let us consider some particular but arbitrarily chosen outcome $\omega$ and restrict to those $t$ that are large enough so that $|X_i(\omega)|\le B_t$ holds for all $i\le t$. Then, $\widehat P_t$ indeed is supported on $[-B_t, B_t]$. Now, at this point we know that our proofs from the bounded support case automatically apply just by simply making the replacement of $[a,b]=[-B_t, B_t]$, $b-a=2B_t$, and with $\KLinf$ replaced by $\KLinf^{(t)}$. In the case where $m>B_t$, we know that the constraint set is empty and thus that $\KLinf^{(t)}(\nu, m)=+\infty$. But for us in our application, $m=\mu$ and $B_t\to\infty$. As a consequence, for all sufficiently large $t$, our constraint indeed is of course feasible. Because we know that our Theorem~\ref{thm:LIL-KLinf}'s upper bound construction uses the affine tilt $\frac{dQ_t}{d\widehat P_t}(x)=1+\theta_t(x-\widehat\mu_t)$ and $\theta_t=\frac{\Delta_t}{\widehat\sigma_t^2}$ and $\Delta_t=(\mu-\widehat\mu_t)_+$. Notice how the only place really where we needed the actual boundedness was in guaranteeing that $1+\theta_t(x-\widehat\mu_t)$ remains positive uniformly over the support. In this situation here, we know that $x\in[-B_t, B_t]$ and $\widehat\mu_t\in[-B_t, B_t]$, and so as a result we get that $|x-\widehat\mu_t|\le 2B_t$. Thus it follows that,
\[
\inf_{x\in[-B_t,B_t]}(1+\theta_t(x-\widehat\mu_t))\ge 1-2B_t\theta_t.
\]
Hence it would suffice that only $2B_t\theta_t\to 0$ almost surely. We know that $\theta_t=\Delta_t/\widehat\sigma_t^2$ and that $\widehat\sigma_t^2\to\sigma^2>0$. The mean LIL tells us that $\Delta_t$ is in fact $O(\sqrt{(\log\log t)/t})$ almost surely along the full sequence (ie, by the limsup). Therefore, we can see that almost surely,
\[
2B_t\theta_t \ =\ 2B_t\frac{\Delta_t}{\widehat\sigma_t^2} \ =\ o\left(\sqrt{\frac{t}{\log\log t}}\right)\cdot O\left(\sqrt{\frac{\log\log t}{t}}\right) \ \xrightarrow[t\to\infty]{}\ 0.
\]
What this tells us is the tilt is going to be feasible for all large $t$ (and certainly supported on $[-B_t, B_t]$). As such, it will be admissible for $\KLinf^{(t)}(\widehat P_t, \mu)$. One more thing we need to account for in this upper bound is the Taylor remainder terms. However, in the bounded proof we know that those Taylor terms scale with powers of the uniform radius $r_t:=\theta_t(b-a)=\theta_t(2B_t)=2B_t\theta_t$. We just showed that $r_t\to 0$ wp 1. So, the same argument gives us the upper asymptotic bound that almost surely,
\[
\KLinf^{(t)}(\widehat P_t, \mu)\le (1+o(1))\frac{\Delta_t^2}{2\widehat\sigma_t^2}.
\]

\noindent Let's now work on the lower bound now to complete our proof. We will use our same lower bound variational argument we used earlier. However, in doing so we will keep track of the remainder in a way that will only depend on the uniform smallness parameter. Take a particular $t$ that's large enough so that $\widehat P_t$ is supported on $[-B_t, B_t]$ and $B_t\ge |\mu|$ (we know this certainly must hold eventually since $B_t$ is sent to $\infty$). Now let us consider some feasible $Q\in\mathcal P([-B_t, B_t])$ where $\int x dQ(x)\ge \mu$ and any $\lambda\in[0, 1/(B_t-\mu))$. Then, exactly as in \eqref{eq:dual-lower-bound}, we know that the same choice of $\varphi_\lambda(x):=\log(1+\lambda(\mu-x))$, we have that,
\[
\KL_{\inf}^{(t)}(\widehat P_t,\mu)\ \ge \int \log\big(1+\lambda(\mu-x)\big)\,d\widehat P_t(x).
\]
Now let us pick a $\lambda=\lambda_t:=\Delta_t/\widehat\sigma_t^2$. Note that for large $t$, we know that this lies in $[0, 1/(B_t-\mu))$. (Because, we know that eventually, $B_t\ge |\mu|$ and that wp 1, $r_t:=2B_t\lambda_t\to 0$. Hence it certainly follows that for all those $t$ large enough that $\lambda_t(B_t-\mu)\le \lambda_t(B_t+|\mu|)\le 2B_t\lambda_t=r_t<1$.) Now, let $X\sim \widehat P_t$ and set $U_t:=\lambda_t(\mu-X)$. Then it follows that for all $x\in[-B_t, B_t]$, $|\mu-x|\le |\mu|+B_t\le 2B_t$. Hence, $|U_t|\le \lambda_t\cdot 2B_t=:r_t$. On our almost sure event now, $\Delta_t = O(\sqrt{(\log\log t)/t})$ and $\widehat\sigma_t^2\to\sigma^2>0$. As a consequence we have that with probability one,
\[
r_t=2B_t\lambda_t=2B_t\frac{\Delta_t}{\widehat\sigma_t^2}\xrightarrow[t\to\infty]{}0.
\]
Meaning, for all $t$ large enough we get that $r_t<1$. In the case where $r_t=0$, we would have that $U_t\equiv 0$, so the inequality would be trivial. Thus when applying Lemma~\ref{lem:Taylor} we can restrict to the case where $r_t\in(0,1)$. Let's do this, ie apply Lemma~\ref{lem:Taylor} with $r=r_t$. Doing so gives us,
\[
\log(1+U_t)\ \ge\ U_t-\frac{U_t^2}{2}-\frac{|U_t|^3}{3(1-r_t)^3}.
\]
Further, we know that since $|U_t|\le r_t$ as we just showed above, it follows that $|U_t|^3\le r_t U_t^2$. Let's now take expectations. Doing so and using all this gives us that, 
\[
\Ebb_{\widehat P_t}[\log(1+U_t)]  \ge \Ebb_{\widehat P_t}[U_t] - \left(\frac{1}{2}+\frac{r_t}{3(1-r_t)^3}\right)\Ebb_{\widehat P_t}[U_t^2].
\]

\noindent Let us look at the relevant case for us, specifically the case where $\Delta_t>0$, so necessarily $\lambda_t>0$. In this case, we have that $\mu-\widehat\mu_t=\Delta_t$. In addition, we know that $\Ebb_{\widehat P_t}[U_t]=\lambda_t(\mu-\widehat\mu_t)=\lambda_t\Delta_t$. And, $\Ebb_{\widehat P_t}[U_t^2]=\lambda_t^2\Ebb_{\widehat P_t}[(\mu-X)^2]=\lambda_t^2(\widehat\sigma_t^2+\Delta_t^2)$. Note that the last step before holds as we did previously, using the fact that $(\mu-X)=(\mu-\widehat\mu_t)+(\widehat\mu_t-X)$ and $\Ebb_{\widehat P_t}[\widehat\mu_t-X]=0$. Let us now plug in $\lambda_t=\Delta_t/\widehat\sigma_t^2$. Doing so gives us that $\Ebb_{\widehat P_t}[U_t]=\frac{\Delta_t^2}{\widehat\sigma_t^2}$ and $\Ebb_{\widehat P_t}[U_t^2]=\frac{\Delta_t^2}{\widehat\sigma_t^2}+\frac{\Delta_t^4}{\widehat\sigma_t^4}$. Therefore we get that,
\[
\Ebb_{\widehat P_t}[\log(1+U_t)] \ \ge\ \frac{\Delta_t^2}{\widehat\sigma_t^2} -\left(\frac{1}{2}+\frac{r_t}{3(1-r_t)^3}\right)\left(\frac{\Delta_t^2}{\widehat\sigma_t^2}+\frac{\Delta_t^4}{\widehat\sigma_t^4}\right).
\]
Note that $r_t\to 0$ and $\Delta_t\to 0$ wp 1 as we know. Therefore as a result of this, the term $\frac{1}{2}+\frac{r_t}{3(1-r_t)^3}$ is simply nothing but $\frac{1}{2}+o(1)$ and the last term with $\Delta_t^4$ of course reduces down $o(\Delta_t^2)$. Hence we get that with probability one,
\[
\Ebb_{\widehat P_t}[\log(1+U_t)] \ge (1-o(1))\frac{\Delta_t^2}{2\widehat\sigma_t^2}.
\]

\noindent Now, recall that for our choice of $\lambda_t$, we had that $\KLinf^{(t)}(\widehat P_t, \mu)\ge \Ebb_{\widehat P_t}[\log(1+U_t)]$. Thus we get almost surely,
\[
\KL_{\inf}^{(t)}(\widehat P_t,\mu) \ge (1-o(1))\frac{\Delta_t^2}{2\widehat\sigma_t^2}.
\]
Now, we already showed the upper bound that $\KLinf^{(t)}(\widehat P_t, \mu)\le (1+o(1))\frac{\Delta_t^2}{2\widehat\sigma_t^2}$. So, we have both an upper and lower version! Hence it must follow that almost surely as $t\to\infty$,
\[
\KL_{\inf}^{(t)}(\widehat P_t,\mu)=\big(1+o(1)\big)\frac{\Delta_t^2}{2\widehat\sigma_t^2}.
\]
We know that the classic mean LIL implies that, as in \eqref{eq:Delta-LIL}, almost surely $\limsup_{t\to\infty}\frac{\Delta_t^2}{(2\sigma^2\log\log t)/t}=1$. Now, we have the almost sure convergence also of $\widehat\sigma_t^2\to\sigma^2$. That tells us that,
\[
\limsup_{t\to\infty}\frac{t}{\log\log t}\cdot\frac{\Delta_t^2}{2\widehat\sigma_t^2}=1,
\]
and so therefore using our proven asymptotic equivalence we get that with probability one,
\[
\limsup_{t\to\infty}\frac{t\,\KL_{\inf}^{(t)}(\widehat P_t,\mu)}{\log\log t}=1,
\]
which completes our proof! \end{proof}

\section{Applications of Theorem~\ref{thm:LIL-growing-envelope}}\label{sec:further-applications-beyond-bounded}
\noindent Now, the reason we know that Theorem~\ref{thm:LIL-growing-envelope} is so powerful is in its reduction of the problem to simply verifying the almost sure envelope $|X_i|\le B_t$ for all $i\le t$ eventually, where $B_t=o(\sqrt{t/\log\log t})$. Such a theorem is only useful if we know it holds for various tail regimes. Hence that is what we will show. The first is for \textit{subgaussian} tails. Meaning, assume that there exists $v>0$ such that for all $x\ge 0$,
\begin{equation}\label{eq:subG-tail}
\Pbb(|X_1-\mu|>x)\le 2\exp\left(-\frac{x^2}{2v}\right).    
\end{equation}
Given this definition of the subgaussian tail, we will present the lemma below for the almost sure envelope for subgaussian random variables.

\begin{lemma}\label{lem:envelope-subG}
Assume that \eqref{eq:subG-tail} holds. Consider any particular $\varepsilon>0$ and define,
\[
u_t:=\sqrt{2(v+\varepsilon)\log t}, \qquad
B_t:=|\mu|+u_t.
\]
Then, almost surely there exists $T<\infty$ such that for all $t\ge T$ and all $1\le i\le t$, we have that $|X_i|\le B_t$. And moreover, $B_t=o(\sqrt{t/\log\log t})$.
\end{lemma}

\begin{proof}
By stationarity, it follows that for each integer $t\ge 3$ we have that,
\[
\Pbb(|X_t-\mu|>u_t)=\Pbb(|X_1-\mu|>u_t)\le 2\exp\left(-\frac{u_t^2}{2v}\right)=2\exp\left(-\frac{v+\varepsilon}{v}\log t\right)=2t^{(-1+\varepsilon/v)}.
\]
Notice that $1+\varepsilon/v>1$, hence it follows that the series $\sum_{t\ge 3} t^{-(1+\varepsilon/v)}$ indeed converges. Meaning, $\sum_{t=3}^\infty \Pbb(|X_t-\mu|>u_t)<\infty$. So by Borel-Cantelli Lemma 1, we can conclude that $\Pbb(|X_t-\mu|>u_t$ infinitely often) must be $0$ indeed. In other words, wp 1 there exists a random $T_0$ such that for all $t\ge T_0$ we have that $|X_t-\mu|\le u_t$. Now, we know that $u_t$ is nondecreasing in $t$, so it follows that for any $t\ge T_0$, $\max_{T_0\le i\le t}|X_i-\mu|\le u_t$. Let us now define the (finite) random constant $M_0:=\max_{1\le i\le T_0}|X_i-\mu|$. We know that $u_t\to\infty$. As a consequence, it follows that there exists $T_1$ such that for all $t\ge T_1$, $u_t\ge M_0$. As such, for all $t\ge \max\{T_0, T_1\}$, we have,
\[
\max_{1\le i\le t}|X_i-\mu|=\max\Big\{\max_{1\le i\le T_0}|X_i-\mu|,\ \max_{T_0\le i\le t}|X_i-\mu|\Big\}\le \max\{M_0, u_t\} = u_t.
\]
So, it follows that for all such $t$ and all $1\le i\le t$, $|X_i|\le |\mu|+|X_i-\mu|\le |\mu|+u_t=B_t$. This proves our envelope property. To finish, note that our $B_t=|\mu|+\sqrt{2(v+\varepsilon)\log t}$ satisfies,
\[
\frac{B_t}{\sqrt{t/\log\log t}}=\frac{|\mu|}{\sqrt{t/\log\log t}}+\sqrt{2(v+\varepsilon)}\frac{\sqrt{\log t}\,\sqrt{\log\log t}}{\sqrt{t}}\xrightarrow[t\to\infty]{}0.
\]
Thus, the takeaway is that indeed $B_t=o(\sqrt{t/\log\log t})$. At this point, we have proved that sub-gaussian tails satisfy both properties. So we are done.
\end{proof}

\noindent Notice how if we take the result of Lemma~\ref{lem:envelope-subG} and Theorem~\ref{thm:LIL-growing-envelope} we get our desired ``same'' LIL statement for subGaussian random variables. That is, constraining the other class to an envelope that grows slowly that eventually contains the data almost surely gives us something beautiful. Something where the empirical $\KLinf$ once again satisfies an exact LIL with a sharp constant $1$. Of course, in addition to subGaussian, we can also choose such a sequence for sub-exponential random variables. Namely, we can construct carefully an almost sure envelope for sub-exponential random variables also.

\begin{lemma}\label{lem:envelope-subexp}
Suppose that there exist constants $K, c>0$ such that for all $x\ge 0$,
\[
\Pbb(|X_1-\mu|>x)\le K e^{-cx}.
\]
Then consider any particular $\varepsilon>0$ and define for $t\ge 3$,
\[
u_t:=\frac{1+\varepsilon}{c}\log t, \qquad B_t:=|\mu|+u_t,
\]
and define the $B_1:=B_2:=B_3$ in a way so that $(B_t)_{t\ge1}$ is nondecreasing. Then, almost surely there must exist $T<\infty$ such that for all $t\ge T$ and all $1\le i\le t$, one has that $|X_i|\le B_t$. In addition, $B_t=o(\sqrt{t/\log\log t})$.
\end{lemma}

\begin{proof}
Before we start our proof, we will point out that the subexponential proof follows the same line very closely as that of the subgaussian. Meaning, to begin, for each integer $t\ge 3$, let's apply stationarity alongside our assume tail bound. Doing so gives us that,
\[
\Pbb(|X_t-\mu|>u_t)=\Pbb(|X_1-\mu|>u_t)\le Ke^{-cu_t}= Ke^{-(1+\varepsilon)\log t}=Kt^{-(1+\varepsilon)}.
\]
Now, the sum $\sum_{t\ge 3}Kt^{-(1+\varepsilon)}<\infty$. By the Borel Cantelli Lemma 1 we therefore once again have that $\Pbb(|X_t-\mu|>u_t)$ infinitely often is $0$. Hence almost surely there must exist a finite $T_0$ such that for all $t\ge T_0$, $|X_t-\mu|\le u_t$. Note that $(u_t)$ is nondecreasing. So, for any $t\ge T_0$ and any $T_0\le i\le t$, we have that $|X_i-\mu|\le u_i\le u_t$. As a consequence we get that $\max_{T_0\le i\le t}|X_i-\mu|\le u_t$. Now let's let $M_0:=\max_{1\le i\le T_0}|X_i-\mu|<\infty$. Again, we send $u_t\to\infty$. As such, there must exist $T_1$ such that $u_t\ge M_0$ for all $t\ge T_1$. So for all $t\ge \max\{T_0, T_1\}$ we have that $\max_{1\le i\le t}|X_i-\mu|\le u_t$. Hence for all $1\le i\le t$,
\(
|X_i|\le |\mu|+|X_i-\mu|\le |\mu|+u_t = B_t,
\)
which proves the envelope property. It remains to show that $B_t=o(\sqrt{t/\log\log t})$. It's easy to see that,
\[
\frac{B_t}{\sqrt{t/\log\log t}} =\frac{|\mu|}{\sqrt{t/\log\log t}}+\frac{1+\varepsilon}{c}\cdot \frac{\log t\cdot \sqrt{\log\log t}}{\sqrt{t}} \xrightarrow[t\to\infty]{}0,
\]
and so that concludes the proof as once again we showed both properties are satisfied, this time wrt subexponential random variables.
\end{proof}

\noindent Lastly, we present an almost sure envelope for any random variable with a finite $p$-th moment.

\begin{lemma}\label{lem:envelope-moment}
Suppose that $\Ebb[|X_1|^p]<\infty$ for some $p>2$. Take any $\gamma>1/p$ and define for $t\ge 3$,
\[
u_t:=t^{1/p}(\log t)^{\gamma}, \qquad
B_t:=\max_{3\le s\le t} u_s,
\]
and define $B_1:=B_2:=B_3$. It follows then that $(B_t)_{t\ge1}$ is deterministic, nondecreasing, and $B_t\to\infty$. In addition, we certainly have that there exists a $T<\infty$ such that for all $t\ge T$, and all $1\le i\le t$, one has that $|X_i|\le B_t$, and $B_t=o(\sqrt{t/\log\log t})$. 
\end{lemma}

\begin{proof}
Once again, for each integer $t\ge 3$, let's apply Markov's inequality and the same stationarity. Doing so gives us that,
\[
\Pbb(|X_t|>u_t)=\Pbb(|X_1|>u_t)\le \frac{\Ebb[|X_1|^p]}{u_t^p} =\frac{\Ebb[|X_1|^p]}{t(\log t)^{p\gamma}}.
\]
We know that $p\gamma>1$. As such, the series $\sum_{t\ge 3}\frac{1}{t(\log t)^{p\gamma}}$ converges. So, $\sum_{t=3}^\infty \Pbb(|X_t|>u_t)<\infty$. So the same Borel Cantelli Lemma 1 tells us that $\Pbb(|X_t|>u_t)$ infinitely often indeed must be zero. In other words, wp 1 there exists finite $T_0$ such that for all $t\ge T_0$, $|X_t|\le u_t$. By our construction $B_t\ge u_t$ and $B_t$ is nondecreasing. Hence it follows that for any $t\ge T_0$ and any $T_0\le i\le t$, $|X_i|\le u_i\le B_i\le B_t$. Hence $\max_{T_0\le i\le t}|X_i|\le B_t$. Let $M_0:=\max_{1\le i\le T_0}|X_i|<\infty$. We know that $B_t\to\infty$. As a consequence, there exists $T_1$ such that $B_t\ge M_0$ for all $t\ge T_1$. Hence for all $t\ge \max\{T_0, T_1\}$, we have that $\max_{1\le i\le t}|X_i|\le B_t$. This is the envelope property! It remains to now show that $B_t=o(\sqrt{t/\log\log t})$. Up to the running maximum, certainly it's the case that $B_t\sim t^{1/p}(\log t)^{\gamma}$ and $p>2$. Therefore,
\[
\frac{B_t}{\sqrt{t/\log\log t}}\le t^{1/p-1/2}(\log t)^{\gamma}\sqrt{\log\log t}\xrightarrow[t\to\infty]{}0.
\]
And so once again, we showed both properties hold, this case with finite $p$ moment random variables. In any case, this concludes the proof. 
\end{proof}
\section{What happens in the $p=2$ case and is the envelope tight?}\label{sec:env-tightness}
In this section, let us explain what would in fact happen if we only assume a weak second moment assumption (ie, $p=2$) on the tail. There are two points that we will detail here that are complementary to one another. First, suppose that a deterministic envelope $(B_t)_{t\ge 1}$ grows strictly faster than $\sqrt{t/\log\log t}$ and eventually contains the data almost surely. If this were to occur, then the time-varying projection cost $\KLinf^{(t)}(\widehat P_t, \mu)$ will be forced onto a scale that is in fact strictly smaller than $\log\log t$. Second, unfortunately if only $\Ebb[X_1^2]<\infty$, then we cannot verify in general the key assumption needed in Theorem~\ref{thm:LIL-growing-envelope}: meaning, we cannot verify the existence of an almost sure valid deterministic envelope with $B_t=o(\sqrt{t/\log\log t})$. In other words, there certainly exist finite variance distributions for which almost surely every envelope will fail. All of this is what we will formalize in this section. To begin, recall our time-$t$ constrained functional defined as,
\[
\KLinf^{(t)}(\nu,m) :=\inf\Big\{\KL(\nu\|Q): Q\in\mathcal P([-B_t,B_t]),\ \int x\,dQ(x)\ge m\Big\}.
\]
And, in the empirical specialization of this if you will where $\nu=\widehat P_t$, we let $\widehat\mu_t=\int x\,d\widehat P_t(x)$ and $\Delta_t:=(\mu-\widehat\mu_t)_+$. With this setup in mind, let us present our first lemma, which will give intuition for what happens when we ``sprinkle'' mass at the boundary.

\begin{lemma}\label{lem:sprinkling-upper}
Take a $B>0$ and let $\nu\in\mathcal P([-B, B])$ have mean $\bar\mu:=\int xd\nu(x)$. Now, for any $m\in(\bar\mu, B)$ define $\varepsilon:=\frac{m-\bar\mu}{B-\bar\mu}\in(0,1)$. In addition, define $Q:=(1-\varepsilon)\nu+\varepsilon\delta_{B}$. Then, it follows that for $Q\in\mathcal P([-B, B])$ we get that $\int xdQ(x)=m$ and $\KL(\nu\|Q)\le -\log(1-\varepsilon)$. And if in fact we have that $\nu(\{B\})=0$ then it follows that $\KL(\nu\|Q)=-\log(1-\varepsilon)$ exactly. Finally as $\varepsilon\downarrow 0$, $-\log(1-\varepsilon)=(1+o(1))\varepsilon$, and in particular if $\varepsilon\le\tfrac12$ then $-\log(1-\varepsilon)\le 2\varepsilon$. 
\end{lemma}

\begin{proof}
The first thing we need to do is show that $Q$ is supported on $[-B, B]$ and has mean $m$. By our own construction we know that $\nu$ is supported on $[-B, B]$ and $\delta_B$ is supported at $B\in[-B, B]$. As a consequence, their convex combination $Q$ must be also. Let's now compute the mean:
\begin{align*}
\int x\,dQ(x) &=\int x\,d\big((1-\varepsilon)\nu+\varepsilon\delta_B\big)(x) = (1-\varepsilon)\int x\,d\nu(x)+\varepsilon\int x\,d\delta_B(x)\\
&= (1-\varepsilon)\bar\mu+\varepsilon B = \bar\mu+\varepsilon(B-\bar\mu)\\
&= \bar\mu+(m-\bar\mu) = m.
\end{align*}
With that shown let us now proceed to show absolute continuity, namely $\nu\ll Q$ and also derive a Radon-Nikodym bound. It's easy to see that for any measurable set $A$ we have that $Q(A)= (1-\varepsilon)\nu(A)+\varepsilon\delta_B(A)\ge (1-\varepsilon)\nu(A)$. Clearly then if $Q(A)$ were $0$, that means that $(1-\varepsilon)\nu(A)\le Q(A)=0$. And so as such, we'd have that $\nu(A)=0$, thereby showing that $\nu\ll Q$ indeed. Now let $L:=\frac{d\nu}{dQ}$ be the Radon-Nikodym derivative. We will show that $L\le \frac{1}{1-\varepsilon}$ $\nu$-almost surely. To begin, simply take an $\eta>0$ and define the measurable $A_\eta:=\left\{L>\frac{1}{1-\varepsilon}+\eta\right\}$. Now using the fact that $\nu(A_\eta)=\int_{A_\eta}LdQ$ and the very definition of $A_\eta$ we have that,
\begin{align*}
\nu(A_\eta) = \int_{A_\eta} L\,dQ >\left(\frac{1}{1-\varepsilon}+\eta\right)Q(A_\eta)\ge\left(\frac{1}{1-\varepsilon}+\eta\right)(1-\varepsilon)\nu(A_\eta)\\
=\big(1+\eta(1-\varepsilon)\big)\nu(A_\eta).
\end{align*}
We know that of course the only way a nonnegative number can be strictly speaking larger than itself multiplied by $1+\eta(1-\varepsilon)>1$ is if it were $0$. So it must be the case that for every $\eta>0$, $\nu(A_\eta)=0$. We know that $\eta>0$ was taken arbitrary and hence we can see that $\frac{d\nu}{dQ}=L\le \frac{1}{1-\varepsilon}$ almost surely under $\nu$. The $\KL$ bound is now immediate. That is, we can use our bound on $L$ and just apply the definition. Namely, $\KL(\nu\|Q)= \int \log\left(\frac{d\nu}{dQ}\right)d\nu\le \int \log\left(\frac{1}{1-\varepsilon}\right)=-\log(1-\varepsilon)\int 1d\nu= -\log(1-\varepsilon)$. As for when we have equality, note that if $\nu(\{B\})=0$ then that would mean that $\delta_B$ would place mass outside of the $\nu$ support. As such on the support of $\nu$ we would have that $Q=(1-\varepsilon)\nu$. Therefore $\frac{d\nu}{dQ}=\frac{1}{1-\varepsilon}$ will hold under $\nu$ wp 1. Plugging this back into the KL definition would therefore give us that $\KL(\nu\|Q)=\int \log\left(\frac{1}{1-\varepsilon}\right)d\nu=-\log(1-\varepsilon)$. We know that as $\varepsilon\downarrow0$, $-\log(1-\varepsilon)\to 0$ and $\varepsilon\to 0$, so it follows that $\lim_{\varepsilon\downarrow 0}\frac{-\log(1-\varepsilon)}{\varepsilon}=1$. So clearly it follows that $-\log(1-\varepsilon)=(1+o(1))\varepsilon$ as $\varepsilon\downarrow0$. It remains to now show the bound we get when $\varepsilon\le \tfrac12$. Let us start by noting that for $0\le u\le \varepsilon$ we have that $1-u\ge 1-\varepsilon$ and thus $\frac{1}{1-u}\le \frac{1}{1-\varepsilon}$. As such,
\[
-\log(1-\varepsilon) =\int_{0}^{\varepsilon}\frac{1}{1-u}\,du \le \int_{0}^{\varepsilon}\frac{1}{1-\varepsilon}\,du =\frac{\varepsilon}{1-\varepsilon}\le 2\varepsilon,
\]
which completes our proof.
\end{proof}

\noindent We are now ready to formalize the idea that large envelopes will collapse the $\log\log t/t$ scale. We will do that with the following proposition.

\begin{proposition}\label{prop:large-envelope-collapse}
Suppose that $X_1, X_2, X_3, \dots$ are all iid with mean $\mu$ and variance $\sigma^2\in(0,\infty)$. Let $(B_t)_{t\ge 1}$ be deterministic and nondecreasing with $B_t\to\infty$ and assume that the envelope event holds. That is,
\[
\mathcal E_B :=\left\{\exists T_B<\infty:\ \forall t\ge T_B,\ \max_{1\le i\le t}|X_i|\le B_t\right\}.
\]
If in addition we have that $\frac{B_t}{\sqrt{t/\log\log t}}\xrightarrow[t\to\infty]{}\infty$, then on the probability-one event where both $\mathcal E_B$ and the classical mean LIL hold we in fact have that
\[
\limsup_{t\to\infty}\frac{t\KLinf^{(t)}(\widehat P_t,\mu)}{\log\log t}=0.
\]
\end{proposition}

\begin{proof}
We're going to begin our proof by first taking a particular outcome $\omega$ that belong to the intersection of both the probability-one events, ie the envelope event $\mathcal E_B$ and the classic mean LIL. So with a very minor abuse of notation, we will omit $\omega$ from our notation. The first thing we need to show is a lower bound on $B_t-\widehat\mu_t$ for large $t$. By the mean LIL, we know that $\widehat\mu_t\to\mu$ and therefore there must exist a $T_\mu<\infty$ and such that for all $t\ge T_\mu$, $|\widehat\mu_t|\le |\mu|+|\widehat\mu_t-\mu|\le |\mu|+1$. In addition, note that $B_t\to\infty$ and $(B_t)$ by definition is nondecreasing. So, there must also exist a $T_B'<\infty$ such that for all $t\ge T_B'$, $B_t\ge 2(|\mu|+1)$. So! That means that for all $t\ge T_0:=\max\{T_\mu, T_B'\}$, it follows that  $B_t-\widehat\mu_t\ge B_t-|\widehat\mu_t|\ge B_t-(|\mu|+1) \ge \frac{B_t}{2}$. Let us now derive a sprinkling bound for $\KLinf^{(t)}(\widehat P_t, \mu)$. Take a $t\ge T_0$ such that we also have $t\ge T_B$, so $\widehat P_t$ is clearly supported $[-B_t, B_t]$. Necessarily if $\Delta_t=0$, then $\widehat \mu_t \ge \mu$ and hence $\KLinf^{(t)}(\widehat P_t, \mu)=0$. So assume that $\Delta_t>0$, or that $\widehat\mu_t<\mu$. Again, since $B_t\to\infty$, there exists $T_1<\infty$ such that $\mu<B_t$ for all $t\ge T_1$. Now considering all of this, for a $t\ge \max\{T_0, T_B, T_1\}$, $\varepsilon_t:=\frac{\mu-\widehat\mu_t}{B_t-\widehat\mu_t}=\frac{\Delta_t}{B_t-\widehat\mu_t}\in(0,1)$. Using the fact that $B_t-\widehat\mu_t\ge B_t/2$, we get that,
\[
\varepsilon_t = \frac{\Delta_t}{B_t-\widehat\mu_t}\le \frac{\Delta_t}{B_t/2}=\frac{2\Delta_t}{B_t}.
\]
Let's now apply Lemma~\ref{lem:sprinkling-upper} with $B=B_t$, $\nu=\widehat P_t$ and $m=\mu$. Doing all this gives us that $\KLinf^{(t)}(\widehat P_t, \mu)\le -\log(1-\varepsilon_t)$. Now, from the mean LIL we know that $\limsup_{t\to\infty}\frac{|\widehat\mu_t-\mu|}{\sqrt{(2\sigma^2\log\log t)/t}}=1$. This means that we conclude that for a particular constant $c:=\sqrt{2\sigma^2}+1>\sqrt{2\sigma^2}$, there exists $T_{\mathrm{LIL}}<\infty$ such that for all $t\ge T_{\mathrm{LIL}}$, $|\widehat\mu_t - \mu|\le c\sqrt{\frac{\log\log t}{t}}$. Therefore, $\Delta_t\le c\sqrt{\frac{\log\log t}{t}}$ for all $t\ge T_{\mathrm{LIL}}$. Therefore, using the fact that $\varepsilon_t\le 2\Delta_t/B_t$ and the bound on $\Delta_t$, it's easy to see that $\varepsilon_t \le \frac{2c}{B_t}\sqrt{\frac{\log\log t}{t}}$. And, because $B_t/\sqrt{t/\log\log t}\to\infty$, we have that,
\[
\frac{1}{B_t}\sqrt{\frac{\log\log t}{t}} =\frac{1}{B_t}\cdot \frac{1}{\sqrt{t/\log\log t}} \;\xrightarrow[t\to\infty]{}\;0.
\]
So, $\varepsilon_t\to 0$. Thus there exists $T_\varepsilon<\infty$ where for all $T\ge T_\varepsilon$, $\varepsilon_t\le \tfrac12$ so Lemma~\ref{lem:sprinkling-upper} gives us that indeed $-\log(1-\varepsilon_t)\le 2\varepsilon_t$. In particular for all sufficiently large $t$ with $\Delta_t>0$ we have that $\KLinf^{(t)}(\widehat P_t, \mu)\le -\log(1-\varepsilon_t)\le 2\varepsilon_t \le 2\cdot \frac{2\Delta_t}{B_t}=\frac{4\Delta_t}{B_t}$. Let's now multiply it all by $t/\log\log t$ and use the fact that $\Delta_t\le c\sqrt{\frac{\log\log t}{t}}$. As such,
\begin{align*}
\frac{t}{\log\log t}\,\KLinf^{(t)}(\widehat P_t,\mu)
&\le \frac{t}{\log\log t}\cdot \frac{4\Delta_t}{B_t}\\
&\le \frac{t}{\log\log t}\cdot \frac{4c}{B_t}\sqrt{\frac{\log\log t}{t}}\\
&=\frac{4c}{B_t}\sqrt{\frac{t}{\log\log t}}.
\end{align*}
Now! We know that $B_t/\sqrt{t/\log\log t}\to\infty$. Therefore, the rhs above will converge to $0$. That indeed proves that,
\[
\limsup_{t\to\infty}\frac{t\,\KLinf^{(t)}(\widehat P_t,\mu)}{\log\log t}=0.
\]
And, hence, we are done.\end{proof}

\noindent With this now shown, it is interesting to note that even with a second moment, we can still get a deterministic almost sure envelope. We formalize that idea with the following lemma.

\begin{lemma}\label{lem:envelope-p2}
Suppose that $\Ebb[X_1^2]<\infty$. Let $(B_t)_{t\ge 1}$ be deterministic and nondecreasing, where $B_t\to\infty$. If in fact we have that $\sum_{t=1}^\infty \frac{1}{B_t^2}<\infty$, then almost surely there exists a $T<\infty$ such that for all $t\ge T$ and all $1\le i\le t$, $|X_i|\le B_t$. In particular we know that for any particular $\gamma>\tfrac12$ and for $t\ge 3$, the choice,
\[
u_t:=\sqrt{t}(\log t)^{\gamma}, \quad B_t:=\max_{3\le s\le t}u_s, \quad B_1:=B_2:=B_3,
\]
satisfies the summability condition and thus is an almost sure valid deterministic envelope.
\end{lemma}

\begin{proof}
Let us start our proof by defining the event $A_t:=\{|X_t|>B_t\}$ for each $t\ge 1$. We know that the sequence is iid by assumption. Hence, of course, $\Pbb(A_t)=\Pbb(|X_1|>B_t)$. By Markov's inequality applied to $X_1^2$ we get that,
\[
\Pbb(A_t)=\Pbb(|X_1|>B_t) = \Pbb(|X_1|^2>B_t^2) \le \frac{\Ebb[X_1^2]}{B_t^2}.
\]
Let us now sum over all $t$ and use the fact that $\sum 1/B_t^2<\infty$. Doing so gives us that,
\[
\sum_{t=1}^\infty \Pbb(A_t)\le \Ebb[X_1^2]\sum_{t=1}^\infty \frac{1}{B_t^2}<\infty.
\]
Therefore, by Borell Cantelli Lemma 1, we can see that $\Pbb(A_t \text{ i.o.})=0$, so wp 1 there exists a finite (possibly random) $T_0<\infty$ such that $|X_t|\le B_t$ for all $t\ge T_0$. At this point, we are now going to ``upgrade'' our envelope to a uniform in $i\le t$ envelope. To begin, consider a particular outcome on which $|X_t|\le B_t$ for all $t\ge T_0$. We know that $(B_t)$ is nondecreasing, we know that for any $t\ge T_0$ and any index $i$ with $T_0\le i\le t$, $|X_i|\le B_i \le B_t$. Now, for finitely many indices $1\le i\le T_0$ let us define $M_0$ as $\max_{1\le i\le T_0}|X_i|<\infty$. We know that $B_t\to\infty$ and so there must exist again a finite (and possibly random) $T_1$ such that $B_t\ge M_0$ for all $t\ge T_1$. Hence for all $t\ge T:=\max\{T_0, T_1\}$ and all $1\le i\le t$, we have that $|X_i|\le B_t$. It remains to now verify our choice of $B_t=\max_{3\le s\le t}u_s$. Note that for $\gamma>\tfrac12$ and $t\ge 3$ we have that $\frac{1}{u_t^2}=\frac{1}{t(\log t)^{2\gamma}}$. By the integral test, we know that the series $\sum_{t\ge 3}\frac{1}{t(\log t)^{2\gamma}}$ converges if and only if $\int_{3}^{\infty}\frac{1}{x(\log x)^{2\gamma}}dx<\infty$. Let $u=\log x$ (and so as such $du=dx/x$). Then,
\[
\int_{3}^{\infty}\frac{1}{x(\log x)^{2\gamma}}\,dx =\int_{\log 3}^{\infty}\frac{1}{u^{2\gamma}}\,du = \left[\frac{u^{1-2\gamma}}{1-2\gamma}\right]_{\log 3}^{\infty}.
\]
This will be finite exactly in the case when $2\gamma>1$, or when $\gamma>\tfrac12$. Now for each $t$, because $B_t\ge u_t$ it follows that we have $\frac{1}{B_t^2}\le \frac{1}{u_t^2}$. Hence, $\sum_{t=1}^\infty \frac{1}{B_t^2}<\infty$, and that concludes our proof at last.\end{proof}

\noindent The importance of Lemma~\ref{lem:envelope-p2} really is in its power on the next corollary. Meaning, the $p=2$ envelope will force the normalized $\limsup$ to be $0$.

\begin{corollary}\label{cor:p2-envelope-limsup0}
Suppose that $\sigma^2\in(0,\infty)$ and $\Ebb[X_1^2]<\infty$. Let $\gamma>\tfrac12$ and let's take $(B_t)$ exactly as in Lemma~\ref{lem:envelope-p2}. Then on the probability one event where both the envelope and mean LIL hold we have that
\[
\limsup_{t\to\infty}\frac{t\,\KLinf^{(t)}(\widehat P_t,\mu)}{\log\log t}=0.
\]
\end{corollary}

\begin{proof}
This is actually a very easy proof to see. Because, by Lemma~\ref{lem:envelope-p2} we know that the envelope event $\mathcal E_B$ will hold almost surely for this particular choice of $(B_t)$. In addition, for $t\ge 3$ we have that $B_t\ge u_t=\sqrt{t}(\log t)^\gamma$. As a consequence we get that,
\[
\frac{B_t}{\sqrt{t/\log\log t}} \ge \frac{\sqrt{t}(\log t)^\gamma}{\sqrt{t/\log\log t}}=(\log t)^\gamma\sqrt{\log\log t}\xrightarrow[t\to\infty]{}\infty.
\]
Hence, clearly the needed conditions are satisfied from Proposition~\ref{prop:large-envelope-collapse}, and so our corollary's conclusion indeed holds and we are done.
\end{proof}

\noindent One last point we will formalize here is that finite variance does not imply an envelope at the LIL scale, and hence that kills the fact that all $p=2$ laws could hold under Theorem~\ref{thm:LIL-growing-envelope}, which is exactly why we consider only $p>2$ in that theorem.

\begin{proposition}\label{prop:p2-no-small-envelope}
There indeed exists a mean-zero distribution with finite variance such that for the deterministic sequence $b_t:=\sqrt{\frac{t}{\log\log t}}$ (where $t\ge 3$), one has that,
\(
\Pbb\big(|X_t|>b_t\ \text{i.o.}\big)=1.
\)
Consequently for any deterministic nondecreasing sequence $(B_t)$ with $B_t=o\left(\sqrt{t/\log\log t}\right)$ we also have that,
\[
\Pbb\big(|X_t|>B_t\ \text{i.o.}\big)=1.
\]
And so, unfortunately, the envelope event $\mathcal E_B$ from Theorem~\ref{thm:LIL-growing-envelope} will fail for this law.
\end{proposition}

\begin{proof}
The very first thing we must do in our proof is build a finite-variance distribution. To that end, let $S$ be a Rademacher random variable, ie that $\Pbb(S=\pm 1)=\tfrac12$. Given this, we will define a nonnegative random variable $Y$ by specifying very clearly its survival function for $y\ge 0$. That is,
\[
\Pbb(Y>y) = 
\begin{cases}
A, & 0\le y<e^e,\\[3pt]\displaystyle
\frac{1}{y^2\log y\,(\log\log y)^2}, &y\ge e^e.
\end{cases}
\]
Note that we choose $A$ of course so that the survival function is continuous at $y=e$. That is,
\[
A:=\frac{1}{(e^e)^2\log(e^e)\,(\log\log(e^e))^2}=\frac{1}{e^{2e}\cdot e\cdot 1^2} = e^{-(2e+1)}\in(0,1).
\]
Clearly, this function is nonincreasing, right-continuous, and tends to $0$ as $y\to\infty$. As a consequence, it certainly defines a valid law on $[0,\infty)$, with an atom at $0$ of size $1-A$. Now, let $Y$ be independent and define the random variable $X:=SY$. Clearly, $X$ is symmetric about $0$. Furthermore, we know that by Cauchy-Schwarz, $\Ebb[|X|]\le \sqrt{\Ebb[X^2]}<\infty$, so certainly the mean exists and by symmetry we get that $\Ebb[X]=0$. With all this being said, we are now going to verify that indeed $\Ebb[X^2]<\infty$. We know that $|X|=Y$, so it must follow that $\Ebb[X^2]=\Ebb[Y^2]$. Note that for a nonnegative random variable by the basic property of the tail integral we have that,
\[
\Ebb[Y^2]=\int_{0}^{\infty}\Pbb(Y^2>s)ds=\int_{0}^{\infty}\Pbb(Y>\sqrt{s})ds.
\]
Let's now make the substitution of $s=y^2$, so clearly $ds=2y\,dy$. As such we get,
\[
\Ebb[Y^2]=2\int_{0}^{\infty}y\Pbb(Y>y)dy.
\]
Let's now compute what this is. Meaning,
\begin{align*}
\Ebb[Y^2] &=2\int_{0}^{e^e} y\,\Pbb(Y>y)\,dy +2\int_{e^e}^{\infty} y\,\Pbb(Y>y)\,dy\\
&=2\int_{0}^{e^e} y\cdot A\,dy +2\int_{e^e}^{\infty} y\cdot \frac{1}{y^2\log y(\log\log y)^2}\,dy\\
&=2A\left[\frac{y^2}{2}\right]_{0}^{e^e} +2\int_{e^e}^{\infty}\frac{1}{y\log y(\log\log y)^2}\,dy\\
&=A(e^e)^2 +2\int_{e^e}^{\infty}\frac{1}{y\log y(\log\log y)^2}\,dy.
\end{align*}

\noindent We now need to evaluate what the remaining integral is. To that end, let us make the substitution $u=\log y$, which means that $du=dy/y$. If we do this we get that,
\[
\int_{e^e}^{\infty}\frac{1}{y\log y(\log\log y)^2}\,dy =\int_{u=e}^{\infty}\frac{1}{u(\log u)^2}\,du.
\]
Let us now substitute again here, this time with $v=\log u$, so therefore $dv=du/u$. Doing this gives us,
\[
\int_{e}^{\infty}\frac{1}{u(\log u)^2}\,du = \int_{v=1}^{\infty}\frac{1}{v^2}\,dv=\left[-\frac{1}{v}\right]_{1}^{\infty} = 1.
\]
So it follows that $\Ebb[Y^2]=A(e^e)^2 + 2<\infty$, so certainly it must be the case that $E[X^2]<\infty$. With this shown, we know need to show that $\Pbb(|X_t|> b_t \text{ i.o.})=1$. To begin, let $X_1, X_2, X_3,\dots$ be iid copies of $X$ and define $E_t:=\{|X_t|>b_t\}$. Of course, $b_t\to\infty$ as $t/\log\log t\to\infty$. As such, there exists $T_0<\infty$ such that $b_t\ge e^e$ for all $t\ge T_0$. For such $t$ we may use the tail formula that,
\[
\Pbb(E_t)=\Pbb(|X_1|>b_t)=\frac{1}{b_t^2\log b_t(\log\log b_t)^2}.
\]
Additionally, because of the fact that $\log\log t\to\infty$, there exists $T_1<\infty$, such that $\log\log t\ge 1$ for all $t\ge T_1$ and hence for all $t\ge T_1$, $b_t^2=\frac{t}{\log\log t}\le t$. As such, $b_t\le \sqrt{t}\le t$. Define $T_2:=\max\{T_0, T_1\}$. Then for all $t\ge T_2$, we have $\log b_t\le \log t$ and $\log\log b_t\le \log\log t$. Hence we get,
\begin{align*}
\Pbb(E_t)
&=\frac{1}{b_t^2\log b_t(\log\log b_t)^2}\\
&=\left(\frac{\log\log t}{t}\right)\cdot \frac{1}{\log b_t}\cdot \frac{1}{(\log\log b_t)^2}\\
&\ge \left(\frac{\log\log t}{t}\right)\cdot \frac{1}{\log t}\cdot \frac{1}{(\log\log t)^2}\\
&= \frac{1}{t\log t\log\log t}.
\end{align*}

\noindent We know that the series $\sum_{t\ge 3}\frac{1}{t\log t\log\log t}$ diverges. This is obvious by the integral test. We proceed in a manner with two substitutions $u=\log x$ and then $v=\log u$. Meaning,
\[
\int_{3}^{\infty}\frac{1}{x\log x\log\log x}\,dx = \int_{\log 3}^{\infty}\frac{1}{u\log u}\,du = \int_{\log\log 3}^{\infty}\frac{1}{v}\,dv = \infty.
\]
As such, $\sum_{t\ge T_2}\Pbb(E_t)=\infty$. We know that the events $E_t$ are independent. So by Borel Cantelli Lemma 2 we have that $\Pbb(E_t\text{ i.o.})=1$. Extending from $b_t$ to any $B_t=o(b_t)$ is easy. We proceed as follows. Let $(B_t)$ be determinstic and nondecreasing as usual, with $B_t=o(b_t)$. Clearly then we have that $B_t/b_t\to 0$, and so there exists $T<\infty$ such that $B_t\le b_t$ for all $t\ge T$. As a consequence, it holds that for all $t\ge T$, $\{|X_t|>b_t\}\subseteq \{|X_t|>B_t\}$. Importantly, if $|X_t|>b_t$ happens infinitely often, then that must mean that $|X_t|>B_t$ does also. As such, since $\Pbb(|X_t|>b_t \text{ i.o.})=1$, it $\Pbb(|X_t|>B_t \text{ i.o.})=1$. So that proves the claim and hence we are finally done.
\end{proof}

\noindent All in all crucially, we have tightness at the scale $\sqrt{t/\log\log t}$. When $\widehat\mu_t<\mu<B_t$, for the empirical instance $\nu=\widehat P_t$ and target $m=\mu$, the sprinkling construction of Lemma~\ref{lem:sprinkling-upper} gives us that $\KLinf^{(t)}(\widehat P_t, \mu)\lesssim \frac{\Delta_t}{B_t}$. However, the local quadratic regime we analyzed in Theorem~\ref{thm:LIL-growing-envelope} gives that $\KLinf^{(t)}(\widehat P_t,\mu)\asymp \Delta_t^2$. We know that $\Delta_t$ will fluctuate at the LIL scale of $\Delta_t\asymp\sqrt{(\log\log t)/t}$. As such, the boundary where $\Delta_t^2$ and $\Delta_t/B_t$ match exactly can be easily computed as follows. That is, $\Delta_t^2\ \asymp\ \frac{\Delta_t}{B_t}$ implies that $B_t\asymp\frac{1}{\Delta_t}\asymp \sqrt{\frac{t}{\log\log t}}$. What is the takeaway? Clearly, this means that all the envelopes that are strictly smaller than $\sqrt{t/\log\log t}$ are going to keep the problem in the quadratic regime. And, they will also give us the sharp constant $1$ as in Theorem~\ref{thm:LIL-growing-envelope}. However, envelopes that are strictly larger than $\sqrt{t/\log\log t}$ will allow sprinkling to dominate. This will collapse the $\log\log t/t$ normalization to $0$ by Proposition~\ref{prop:large-envelope-collapse}.

\end{document}